\documentclass[11pt,a4paper]{article}

\usepackage{soul}

\usepackage{amssymb,amsmath, amsthm}
\usepackage[usenames,dvipsnames]{xcolor}
\usepackage{color}
\usepackage{tikz}
\usepackage{pgfplots}
\usetikzlibrary{decorations.pathmorphing,patterns}
\usepackage{floatrow}
\usepackage[small]{caption}
\newtheorem{thm}{Theorem}[section]
\newtheorem{lem}[thm]{Lemma}

\newtheorem{prop}[thm]{Proposition}
\newtheorem{cor}[thm]{Corollary}
\theoremstyle{definition}
\newtheorem{defn}[thm]{Definition}
\newtheorem{example}[thm]{Example}
\newtheorem{rem}[thm]{Remark}

\newcommand\redgrey{red for colour online/pdf version, dark grey for black and white print}

\def\cA{{\mathcal A}}
\def\cD{{\mathcal D}}
\def\cJ{{\mathcal J}}

\newcommand{\ds}{\displaystyle}
\newenvironment{mat}{ \left[ \begin{array}{ccccc} }{\end{array} \right]  }
\newenvironment{vek}{ \left[ \begin{array}{c} }{\end{array} \right] }

\newcommand{\dom}{{\cal D}}
\newcommand{\ran}{{\cal R}}

\newcommand{\la}{\langle}
\newcommand{\ra}{\rangle}

\newcommand{\hhalf}{H_{\frac{1}{2}}}
\newcommand{\half}{\frac{1}{2}}

\newcommand{\be}{\begin{equation}}
\newcommand{\ee}{\end{equation}}
\newcommand{\bea}{\begin{eqnarray}}
\newcommand{\eea}{\end{eqnarray}}
\newcommand{\beann}{\begin{eqnarray*}}
\newcommand{\eeann}{\end{eqnarray*}}

\newcommand{\re}{{\rm Re}\hspace{0.08mm}}
\newcommand{\im}{{\rm Im} \hspace{0.08mm}}
\renewcommand{\Re}{\operatorname{Re}}
\renewcommand{\Im}{\operatorname{Im}}

\newcommand{\V}{\Vert}

\newcommand\rmref[1]{{\rm\ref{#1}}}
\newcommand\vsqueeze[1]{\nopagebreak\vspace{-#1}}

\newcommand\C{{\mathbb C}}
\newcommand\R{{\mathbb R}}
\newcommand\N{{\mathbb N}}
\renewcommand\i{{\rm i}}
\renewcommand\V{V\kern-0.1em}

\title{\bf \Large Numerical Range and Quadratic Numerical Range\\ for Damped Systems}%
\author{Birgit Jacob, Christiane Tretter, Carsten Trunk, and Hendrik Vogt}%

\begin{document}
\maketitle

\abovedisplayskip=8pt plus 3pt minus 3pt
\belowdisplayskip=\abovedisplayskip
\abovedisplayshortskip=-1pt plus 3pt
\belowdisplayshortskip=\belowdisplayskip

\begin{abstract}%
We prove new enclosures for the spectrum of non-selfadjoint operator matrices associated with second order linear differential equations
$\ddot{z}(t) + D \dot{z} (t) + A_0 z(t) = 0$ in a Hilbert space. Our main tool is the quadratic numerical range for which we establish
the spectral inclusion property under weak assumptions on the operators involved; in particular, the damping operator only needs to be
accretive and may have the same strength as $A_0$.   By means of the quadratic numerical range, we establish tight spectral estimates
in terms of the unbounded operator coefficients $A_0$ and $D$ which improve earlier results for sectorial and selfadjoint $D$;
in contrast to numerical range bounds, our enclosures may even provide bounded imaginary part of the spectrum or a spectral free vertical strip.
An application to small transverse oscillations of a horizontal pipe carrying a steady-state flow of an ideal incompressible fluid illustrates that our new bounds are explicit.
\end{abstract}

\indent
\phantom{\bf Keywords} matrix, numerical range, quadratic numerical range.\\



\section{Introduction}

Many linear stability problems in applications, in particular in elasticity theory and hydromechanics, are modeled by second order differential equations of the form
\begin{equation}
\label{sys}
\ddot{z}(t) + D \dot{z} (t) + A_0 z(t) = 0
\end{equation}
in a Hilbert space $H$ where $A_0$ is a self-adjoint and uniformly positive operator in $H$ and $D$ is a linear operator in $H$ representing e.g.\ the damping of the underlying system.
Here we consider the case that $A_0^{-\frac 12}D A_0^{-\frac 12}$ is bounded and accretive. Both $A_0$ and $D$ may be unbounded, 
$D$ may be equally strong as $A_0$ and need not be self-adjoint, and for some results, $D$ need not even be sectorial.

By means of the standard substitution $x = (z, \dot z)^{\top}$, the second order differential equation \eqref{sys} is equivalent to a first-order system
\begin{equation}
\label{first1}
\dot{x} (t)  =  \cA x(t)
\end{equation}
in a suitably defined product Hilbert space. More precisely, if we equip the \vspace{-2mm} space $H_{\half}:=\cD(A_0^{\half})$ with the graph norm of $A_0^{\half}$, then the operator
$\mathcal A:\cD(\mathcal A)\subset H_{\half} \times H\rightarrow H_{\half}\times H$ associated with \eqref{sys} is defined as
\begin{equation} \label{AA}
\mathcal A = \left[\begin{array}{cc} 0 & I \\ -A_0 & -D
\end{array} \right],
\quad
\mathcal D(\mathcal A) = \left\{ \ds \begin{bmatrix} z \\ w \end{bmatrix}  \in H_{\half} \times H_{\half}
  \mid A_0 z + D w \in H  \right\} .
\end{equation}

Under stronger assumptions on the damping operator $D$ such as self-adjoint\-ness and/or stronger relative boundedness, operators of this form and applications in elasticity theory 
or hydromechanics have been studied intensively in the literature for more
than 20 years, see
e.g.\ \cite{biw,baen,chru,CLL,chtr,HSII,hu2,huNEU,JT,JTW,lanshk,shk94,TWII,ves}.
In particular, it was proved that $\cA$ is boundedly invertible, has spectrum in the closed left half-plane, and
generates a strongly continuous semigroup of contractions on $H_{\half} \!\times\! H $, see e.g.\ \cite[Proposition 5.1]{TWII}.

Another example for a differential equation \eqref{sys} and corresponding operator $\cA$ are abstract Klein-Gordon equations originating in quantum mechanics, see e.g.\ \cite{LNT2} and the references therein. In this case $A_0$ has the form $A_0=H_0-V^2$ and $D=2V$ where $H_0$ is a self-adjoint uniformly positive operator, e.g.\ $-\Delta+mc^2$ on $\R^n$ with particle mass $m>0$, and
$V$ is a symmetric operator such that $VH_0^{\smash[t]{-\half}\vphantom{\half}}$ is bounded. By means of indefinite inner product methods, the spectrum of $\cA$ was analyzed and criteria on~$D$ were found ensuring that $\cA$ generates a group of bounded unitary operators in a Pontryagin space in \cite{LNT1}.

The aim of this paper is to establish new enclosures for the spectrum of the operator $\cA$ in \eqref{AA} under rather weak assumptions on the damping operator $D$, allowing it to be as strong as $A_0$ so that even very general perturbation results such as \cite{CueT} do not apply.
To this end, we do not only use the classical numerical range of $\cA$, but also the so-called \emph{quadratic numerical range}. The latter
 was introduced in 1998
for operator matrices with bounded off-diagonal entries in \cite{LT}, shortly after studied in great detail for bounded operator matrices in
\cite{LMMT,LMT}, and in 2009 generalized to diagonally dominant and off-diagonally dominant operator matrices in \cite{T-art}.
Unlike the numerical range, the quadratic numerical range is not convex: it may consist of two components which need not be convex either. 
Since the quadratic numerical range is always contained in the numerical range, see \cite{T}, it may give tighter spectral enclosures. 

We show that this is indeed \emph{always} the case here, for uniformly accretive, for sectorial and even for self-adjoint damping operator $D$ 
(see Figures~\ref{fig0}-\ref{fig9} below). 

If $D$ is only assumed to be uniformly accretive relative to $A_0$ in $H$ (and hence uniformly accretive in $H$) and its numerical range is e.g.\ a half-plane, then the numerical range cannot provide a better spectral enclosure than the left half-plane since it is convex
and contains the numerical ranges of the diagonal elements $D$ and $0$ of $\cA$. The quadratic numerical range yields a non-convex enclosing set to the left of the imaginary axis and, under a certain additional condition, it provides a vertical strip free of spectrum, see Theorem \ref{hendrik}.

If $D$ is assumed to be sectorial with angle $<\pi$ and uniformly accretive in $H$, 
then the quadratic numerical range~of $\cA$ always yields an enclosure 
with corner at $0$, whereas the numerical range may still be a half-plane; if $D$ is uniformly accretive relative to $A_0$ in $H$,
the former is even contained in a sector, while the latter only gives a parabolic enclosure, see Theorem \ref{supernett} and Proposition \ref{SpekInclusion}.
In fact, it was proved in \cite{LMT} for the bounded case that, while every corner of the numerical range
must belong to the spectrum $\sigma(\cA)$, corners of the quadratic numerical range may also belong to the spectrum of a diagonal entry of $\cA$. 
Here $0\notin\sigma(\cA)$, but $0$ belongs to both numerical range and quadratic numerical range; hence $0$ cannot be a corner of $W(\cA)$, but $0$ may be, and indeed is, a corner of $W^2(\cA)$ since
it belongs to the spectrum of the zero operator on the diagonal of~$\cA$.

Even for self-adjoint $D$, the difference between numerical range and quadratic numerical range is substantial. Whereas the imaginary part of the numerical range is always unbounded if
$A_0$ is unbounded, 
the quadratic numerical range may have bounded imaginary part,
may be partly confined to the negative real axis or may even be entirely real, see Theorem \ref{nett}! In the latter case, under a certain additional condition, it may even consist of two disjoint real intervals.

There are two key problems we have to solve before we can take advantage of the quadratic numerical range. Firstly, the operator $\mathcal A$ in (\ref{AA}) is not an operator matrix itself since its domain does not decompose according to the decomposition $H_{\half}\times H$ of the space in which $\cA$ acts; in fact, $\cA$ is merely the closure of the operator matrix $\cA|_{H_1\times H_1}$ and only the quadratic numerical range of $\cA|_{H_1\times H_1}$ is defined. Secondly, the operator matrix $\cA|_{H_1\times H_1}$ with its three unbounded entries $I: H \to H_{\half}$, $A_0: H_{\half} \to H$, and $D:H\to H$,
is neither diagonally dominant nor off-diagonally dominant; in fact, in the first column the stronger entry is the off-diagonal $A_0$, while in the second column the stronger entry is the diagonal~$D$.

Our first main result is the 
so-called spectral inclusion property of the quadratic numerical range,
i.e.\ the set of inclusions
\begin{alignat*}{2}
  \sigma_p(\cA|_{H_1\times H_1}) & \subset W^2(\cA|_{H_1\times H_1}), \quad  & \sigma_{ap}(\cA|_{H_1\times H_1}) & \subset \overline{W^2(\cA|_{H_1\times H_1})},\\
  \sigma_p(\cA) & \subset \overline{W^2(\cA|_{H_1\times H_1})}, & \sigma_{ap}(\cA) & \subset \overline{W^2(\cA|_{H_1\times H_1})},
\end{alignat*}
for the point and approximate point spectrum of $\cA|_{H_1\times H_1}$ and $\cA$, respectively. As usual,
one has to require the existence of at least one point of the resolvent set $\rho(\cA)$ in each component of $\C\setminus \overline{W^2(\cA|_{H_1\times H_1})}$
to obtain the full chain of spectral enclosures
\begin{equation}
\label{chain}
\sigma(\mathcal A) \subset \overline{W^2(\cA|_{H_1\times H_1})} \subset \overline{W(\mathcal A)}.
\end{equation}

Although neither the numerical range nor the quadratic numerical range may be
determined precisely, analytic estimates for either of them provide bounds
for the spectrum via the enclosures~\eqref{chain}. We derive an estimate for $W(\cA)$ and a series of estimates for $W^2(\cA|_{H_1\times H_1})$
in terms of various constants relating the ``real part" of the operator $D$ to $A_0$ and, if $D$ is sectorial with angle $<\pi$, in terms of its sectoriality angle.
In all cases, the quadratic numerical range yields tighter bounds than the numerical range since it allows for finer estimates.
Moreover, we compare all the obtained estimates for $W^2(\cA|_{H_1\times H_1})$ and combine them to further improve the enclosure for the spectrum. 

As an application and illustration of our results, we consider an operator of
the form (\ref{AA}) arising in the investigation of small transverse
oscillations of a pipe carrying steady-state flow of an ideal incompressible
 fluid. The corresponding second order equation \eqref{sys} is of the form
\begin{equation}\label{pipe}
   \frac{\partial ^2 u } {\partial t^2 } + \frac{\partial
   ^2}{\partial r^2 } \left [ E  \frac{\partial ^2 u } {\partial r^2
   } + {C} \frac{\partial ^3 u }{\partial r^2 \partial t }
   \right] + K \frac{\partial^2 u}{\partial t\partial x} =0,
   \hspace{2em} r \in (0,1), \ t > 0,
\end{equation}
where $u(r,t) $ denotes the transverse oscillation at time $t$ and position $r$, and $E$, $C$, $K$ are positive physical constants.
Here both operator coefficients $A_0 =\frac{\partial^2}{\partial r^2} E \frac{\partial^2}{\partial r^2}$ and $D = \frac{\partial^2}{\partial r^2} C \frac{\partial^2}{\partial r^2}+ K \frac{\partial}{\partial r}$ in $L^2(0,1)$
with appropriate domains are fourth order differential operators and hence have the same strength.
We use this example to show that all constants involved in our abstract results may be determined analytically and we establish a new enclosure for the spectrum of this problem.
In particular, we derive a threshold for the damping constant $C$ at which a spectral free strip opens up (see Figures~\ref{fig8}, \ref{fig9} below).

Throughout this paper we use the following notation. For a
closable densely defined linear operator
$S$ in some Banach space $X$ we denote by $\rho(S)$ the resolvent set, by $\sigma_p(S)$
the point spectrum, and by $\sigma_{ap}(S)$ the {\em approximate point
spectrum}, i.e.\ the set of all $\lambda\in \C$ for which there is a sequence
$(x_n)_{n\in \mathbb N}$ in $\dom(S)$ such that
\[
\|x_n\| =1, \quad \|(S-\lambda )x_n\|\rightarrow 0, \quad n \to \infty,\]
see e.g.\ \cite[p.\ 242]{EN}. Clearly, the point spectrum is a
subset of the approximate point spectrum; moreover, the boundary of
the spectrum $\sigma(S)$ belongs to $\sigma_{ap}(S)$,
see e.g.\ \cite[IV \S 1.10]{EN}.

\section{Operator framework}\label{section2}

In this section, we rigorously introduce the operator $\cA$ in \eqref{AA} associated with the second order differential equation \eqref{sys}.
Throughout this paper $H$ is a Hilbert space and we assume the following. 
\ \\

{\bf (A1)} The operator $A_0 : \dom (A_0) \subset H \rightarrow H$ is a self-adjoint and uniformly positive linear operator
on $H$ such that $0$ is in the resolvent set of $A_0$.

\ \\
Assumption (A1) allows us to introduce
Hilbert spaces $ H_{\frac{1}{2}}$ and $ H_{-\frac{1}{2}}$ by means of $A_0$ as follows.
We define
\[
H_{\frac{1}{2}} :=  \dom(A_0^{\frac{1}{2}}), \quad
\| \cdot \|_{\hhalf} \!:= \|A_0^{\frac{1}{2}} \cdot \|_{H},
\]
and we set $H_{-{\frac{1}{2}}}\!:=H_{\frac{1}{2}}^*$, where the duality is taken with
respect to the pivot space~$H$; in other words, $H_{-{\frac{1}{2}}}$ is the completion of $H$ with respect to
\[
\smash{\|z\|_{H_{-{\frac{1}{2}}}}= \|A_0^{-{\frac{1}{2}}}z\|_H.}
\]
If we further define
$
H_1:= \dom(A_0) \mbox{ with the norm }\| \cdot \|_{H_{1}} \!:= \|A_0 \cdot \|_{H}
$,
then $A_0$ may be viewed as a bounded operator $A_0:H_1 \to H$ and extends to a bounded operator $A_0:H_{\frac{1}{2}} \rightarrow H_{-{\frac{1}{2}}}$; in both cases we
keep the notation $A_0$.

If we denote the inner product on $H$ by $\langle\cdot,\cdot\rangle_H$
or $\langle\cdot,\cdot\rangle$ and the duality pairing on $H_{-{\frac{1}{2}}}\times H_{{\frac{1}{2}}}$ by
$\smash{\langle\cdot,\cdot\rangle_{H_{-{\frac{1}{2}}}\times H_{{\frac{1}{2}}}}}$, then,
for $(z',z)^{\top}\in H\times \smash{H_{\frac{1}{2}}}$,
\[
\langle z',z\rangle_{H_{-{\frac{1}{2}}}\times H_{{\frac{1}{2}}}}=\langle
z',z\rangle_H.
\]

{\bf (A2)} The (damping) operator $D: \smash{H_{\frac{1}{2}}} \to \smash{H_{-\frac{1}{2}}}$ is bounded
and the (bounded) operator $A_0^{-\frac 12}D A_0^{-\frac 12}$ is accretive in $H$, i.e.\ 
\[
 \re \langle Dz, z\rangle_{H_{-\frac{1}{2}}\times
  H_{\frac{1}{2}}} \ge 0, \qquad z\in H_{\frac{1}{2}}\,.
\]

{\bf (A3)} The operator $D$ maps the space $H_1=\dom(A_0)$ into $H$.


\begin{rem}
\label{DD}
The bounded operator $D:H_{\half} \to H_{-\half}$ has $H_1$ as a core. If we view $D$ as an operator in $H$ with domain $H_1$, then it is densely defined and accretive by (A2), (A3),
\[
  W(D):= \{ \langle Dg,g \rangle \mid g \in H_1,\ \|g\|=1 \} \subset \{ z\in \C \mid \Re z \ge 0 \},
\]
hence closable by \cite[Theorem~\V.3.2]{K}. In the following, we use the notation $D$ for both~operators.
\end{rem}

%

In the product Hilbert space $H_{\frac{1}{2}}\times H$ we now consider the operator
$\cA:\dom(\cA)\subset H_{\frac{1}{2}}  \times H\rightarrow H_{\frac{1}{2}}  \times H$ given by
\begin{equation}
\label{AA1}
\cA = \begin{mat} 0 & I \\
\!-A_0 & -D\! \end{mat}, \quad
\dom(\cA) = \left\{ \ds \begin{bmatrix} z \\ w \end{bmatrix} \in H_{\frac{1}{2}} \times H_{\frac{1}{2}} \mid A_0 z + D w \in H  \right\} .
\end{equation}

\vspace{0mm}

\begin{prop}
\label{Ainv}
The operator $\cA$ is closed with bounded inverse given~by
\begin{equation}
\label{NaDann}
 \cA^{-1}=\begin{mat} -A_0^{-1}D & -A_0^{-1} \\ I & 0 \end{mat}
\end{equation}
in $H_{\frac{1}{2}}\times H$, $\cA$ generates a $C_0$-semigroup of contractions,
and $H_1 \times H_1$ is a core of $\cA$, i.e.\ 
\[ 
\overline{{\cA}|_{H_1\times H_1}} = \cA.
\]
Moreover, $\cA$ is bounded if and only if $A_0$ is a bounded operator in $H$.
\end{prop}

\begin{proof}
The formula \eqref{NaDann} for the inverse of $\cA$ is easy to check;
it is also easy to see that all entries therein are bounded operators between the respective Hilbert spaces.
Hence $\cA$ is a closed operator. The semigroup property was shown e.g.\ in \cite{HS0}.

By (A3), we have $H_1 \times H_1 \subset \dom(\cA)$.
Hence $\overline{{\cA}|_{H_1\times H_1}} \subset {\cA}$ since ${\cA}$
is a closed operator. Thus
it remains to be shown that ${\cA}\subset \overline{{\cA}|_{H_1\times
H_1}}$.
Let $(z,w)^{\top} \in \dom ({\cal A})$, i.e.\ $z, w\in H_{\frac 12}$ and $f:=A_0z+Dw\in H$.
Since $H_1$ is dense in $H_{\frac 12}$, there exists a sequence
$(w_n)_{n\in\mathbb N}$ in $H_1$ such that $w_n \!\to\! w$, $n\!\to\!\infty$, in $H_{\frac{1}{2}}$.
If we define $z_n:= A_0^{-1}f-A_0^{-1}Dw_n\in H_1$, $n\in\mathbb N$, then
$z_n\to z$, $n\to \infty$, in $H_{\frac{1}{2}}$ and $A_0z_n+Dw_n=f$ in $H$. This shows that
$(z_n,w_n)^{\top} \in \dom(\cA)$ and $\bigl(\cA (z_n,w_n)^{\top}\bigr)_{n\in\N} = (w_n,-f)^{\top}$ converges in $H_{\frac{1}{2}}\times H$.

Clearly, if $\cA$ is bounded in $H_{\frac{1}{2}}\times H$,
then so is $A_0 : H_{\frac 12}\to H$. This is equivalent \vspace{-1.5mm} to $A_0^{\frac 12}: H \to H$ being bounded which implies that $A_0$ is bounded in $H$.
Vice versa, 
the boundedness of $A_0$ implies that $\cD(A_0) = H$ and $H_{\frac 12} = H = H_{-\frac 12}$ with all norms being equivalent. Then also the entries $I: H \to H_{\frac 12}$ and $D:H\to H$ in $\cA$ are bounded and hence so is $\cA$.
\end{proof}

\begin{rem}\label{spektrum}
Proposition~\rmref{Ainv} implies that $\sigma(\cA)$ is contained in the closed left half-plane and that $0\in \rho(\cA)$. However, otherwise the spectrum of $\cA$ may be quite arbitrary, see {\rm \cite[Example~3.2]{JT}}.
\end{rem}

In the following sections we will establish new, tighter enclosures for the spectrum of $\cA$ in terms of its entries $A_0$ and $D$; particular attention will be paid to the case of sectorial and self-adjoint damping operator $D$.


\section{The numerical range of $\cA$}
\label{section3}

In this section we investigate the numerical range of the operator $\cA$ in \eqref{AA1}, which is defined as
\[
  W(\cA) := \bigl\{ \la \cA x, x \ra_{H_{\frac{1}{2}}\times H}  \mid x \in \dom(\cA),\ \|x\|=1 \bigr\}.
\]
By the Toeplitz-Hausdorff Theorem, $W(\cA)$ is always a convex subset of~$\C$, see \cite[Theorem~\V.3.1]{K},
and it has the so-called spectral inclusion property
\begin{equation}
\label{nr-specincl}
 \sigma_p(\cA) \subset W(\cA), \quad \sigma_{ap}(\cA) \subset \overline{W(\cA)}.
\end{equation}
Since $\cA$ is unbounded in general, additional assumptions are needed to ensure $\sigma(\cA) \subset \overline{W(\cA)}$:
if a component $\Omega$ of $\C\setminus \overline{W(\cA)}$ contains a point of $\rho(\cA)$, then $\Omega \subset \rho(\cA)$,
see \cite[Theorem~\V.3.2]{K}.

\smallskip

The following constants will play an important role throughout this paper, see also \cite{JT,JT2}:
\begin{equation}
\label{bgdm}
\begin{split}
\beta &:= \inf_{z\in H_{\frac{1}{2}}\setminus\{0\}}  \dfrac{\re \langle Dz, z\rangle_{H_{-\frac{1}{2}}\times
  H_{\frac{1}{2}}}}{\|z\|^2} \in [0,\infty),\\
\gamma &:= \sup_{z\in H_{\frac{1}{2}}\setminus\{0\}}  \dfrac{\re \langle Dz, z\rangle_{H_{-\frac{1}{2}}\times
  H_{\frac{1}{2}}}}{\|z\|^2}\in [0,\infty],\\
\delta &:= \inf_{z\in H_{\frac{1}{2}}\setminus\{0\}}  \dfrac{\re \langle Dz, z\rangle_{H_{-\frac{1}{2}}\times
  H_{\frac{1}{2}}}}{\|z\|^2_{\hhalf}}\in [0,\infty),\\
\mu &:= \inf_{z\in H_{\frac{1}{2}}\setminus\{0\}}  \dfrac{\re \langle Dz, z\rangle_{H_{-\frac{1}{2}}\times
  H_{\frac{1}{2}}}}{\|z\| \|z\|_{\hhalf}}\in [0,\infty).
\end{split}
\end{equation}

By (A1), the operator $A_0$ is uniformly positive, i.e.\  there exists a constant $a_0>0$ such that
$\la A_0 z,z\ra \ge a_0^2 \|z\|^2$ for $z\in \cD(A_0)$. In other words,
\begin{equation}
\label{a_0}
\smash{\|z\|_{\hhalf} = \|A_0^{\frac{1}{2}}z\| \geq a_0 \|z\|, \quad z \in H_{\frac{1}{2}}\,;}
\end{equation}
note that one may choose $a_0 = (\min \sigma(A_0))^{\frac 12} = \|\smash{A_0^{-\frac{1}{2}}}\|^{-1}$.
Altogether, we have the following estimates between the constants in \eqref{bgdm}:
\begin{equation}
\label{ineq}
\mu^2 \ge \beta \delta, \quad \gamma \ge \beta \ge a_0 \mu,  \quad \mu \ge {a_0}\delta.
\end{equation}

Note that $\beta>0$ means that $D$ is uniformly accretive as an operator in~$H$ with domain~$H_1$, $\Re W(D) \ge \beta$,
while $\delta>0$ means that $D$ is uniformly accretive relative to $A_0$ in~$H$, i.e.\ the numerical range of the linear operator pencil $L$ in $H$, cf.\ \cite{LMT}, given by $L(\lambda):= D-\lambda A_0$, $\dom (L(\lambda))=H_1$, for $\lambda \in \C$ satisfies $\Re W(L) \ge \delta$.
Note that both $\mu>0$ and $\delta>0$ imply $\beta>0$.

\vspace{1mm}

The following three simple observations will be useful in the following.


\begin{rem}
\label{0lemma}
In the definition \eqref{bgdm} of $\beta$, $\gamma$, $\delta$, $\mu$, the infimum and supremum, respectively, may equivalently be taken over $z\in H_1 \setminus\{0\}$ since $H_1$ is a core for $D$, see Remark \ref{DD}.
\end{rem}

\begin{lem}
\label{0alemma}
If $\gamma<\infty$ and $\mu>0$, 
then $A_0$ is a bounded operator in~$H$ with 
$\|A_0\| \le \frac{\gamma^2}{\mu^2}$;
the same holds if $\gamma<\infty$ and $\delta>0$.
\end{lem}

\begin{proof}
By definition of $\mu$, we have
$\re \la Dz,z\ra \ge \mu \|z\| \smash{\|z\|_{\hhalf}}$ for all $z\in H_1 \setminus \{0\}$ and hence, because $\gamma<\infty$ and $\mu>0$,
\[
  \|A_0^{\frac 12} z\| = \|z\|_{\hhalf} \le \frac 1\mu
  \smash{\frac{\re \langle Dz, z\rangle}{\|z\|^2} \|z\| \le \frac \gamma\mu \|z\|;}
\]
note that $\delta>0$ implies $\mu>0$ by \eqref{ineq}.
\end{proof}

\begin{rem}
\label{numericalrange}
We have $\lambda \in W(\cA)$ if and only if there is $(f, g)^{\top} \!\in \dom(\cA)$,
$\smash{\|f\|^2_{\hhalf}} \!+ \|g\|^2 = 1$, with
\begin{equation}\label{WA}
\begin{split}
\lambda =
\left\la \begin{mat} 0 & I \\
-A_0 & -D \end{mat} \right.&\left. \!\! \begin{vek}f\\g
\end{vek}, \begin{vek}f\\g
\end{vek}\right\ra_{H_{\frac{1}{2}}\times H}\! = \la g,f \ra_{H_{\frac 12}} - \la A_0f\!+\!Dg,g\ra
\\ & =
\la g,f\ra_{\hhalf} - \la A_0 f,g\ra_{H_{-\frac{1}{2}}\times  H_{\frac{1}{2}}} - \la
Dg,g\ra_{H_{-\frac{1}{2}}\times
  H_{\frac{1}{2}}}\\ & = -2 \i \,
\im\la A_0^{\frac{1}{2}} f,  A_0^{\frac{1}{2}} g\ra- \la Dg,g\ra_{H_{-\frac{1}{2}}\times
  H_{\frac{1}{2}}}.
\end{split}
\end{equation}
\end{rem}

\begin{prop}
\label{1lemma}
\begin{itemize}
\item[\rm (i)]
The numerical range $W(\cA)$ of $\cA$ is contained in the closed
left half-plane and
\begin{equation}
\label{WAincl}
W(-D) \cup \{0\} \subset W(\cA).
\end{equation}
\item[{\rm (ii)}] The real part $\Re W(\cA)$ satisfies
\begin{equation}\label{oberwolfach}
\inf \,( \Re W(\cA) )= -\gamma, \quad \max \,(\Re W(\cA)) = 0;
\end{equation}
in particular, $\Re W(\cA)$
is bounded if and only if $\gamma < \infty$.
\item[\rm (iii)]
The imaginary part $\Im W(\cA)$ 
is bounded if and only if $A_0$ is a bounded operator in $H$;
in this case, also $D$ is bounded and
\[ \mathopen| \Im W(\cA) | \le \|A_0^{\frac 12}\| + \mathopen\|\Im D\|. \]
\end{itemize}
\end{prop}
\vspace{-2mm plus 1mm}

\begin{proof}
(i) By Remark~\ref{numericalrange}, assumption (A2) ensures that $W(\cA)$ is contained in the closed left half-plane.

\begingroup
\thickmuskip=4mu plus 4mu
If we choose $g=0$ and $f\in H_1$ with $\smash{\|f\|_{\hhalf}}\rule[-1ex]{0pt}{0pt}=1$, then $(f,0)^{\top} \mkern-2mu\in \dom(\cA)$ and
Remark~\ref{numericalrange} shows that $0\in W(\cA)$. If we choose $f=0$ and $g\in H_1$ with $\|g\|=1$, then $(0,g)^{\top} \mkern-2mu\in \dom(\cA)$ and
Remark~\ref{numericalrange} shows that $-\la Dg,g\ra\in W(\cA)$.
\endgroup

(ii) The second equality in \eqref{oberwolfach} is immediate from (i).
The inclusion \eqref{WAincl} implies that $\inf \,( \Re W(\cA) )\le -\gamma$,
cf.\ Remark~\ref{0lemma};
the opposite inequality
follows since for $\lambda \in W(\cA)$, by Remark~\ref{numericalrange}, 
there exists $g\in H_{\frac 12}$ with $\|g\|\le 1$ such that either $\lambda=0 \ge - \gamma$ if $g=0$ or else, if $g\ne 0$,
\[
\re \lambda =  -\|g\|^2 \frac{\re \langle Dg, g\rangle_{H_{-\frac{1}{2}}\times
H_{\frac{1}{2}}}}{\|g\|^2} \geq \, -\!
\sup_{z\in H_{\frac{1}{2}}\setminus\{0\}} \! \frac{\re
\langle Dz, z\rangle_{H_{-\frac{1}{2}}\times
H_{\frac{1}{2}}}}{\|z\|^2} =-\gamma.
\]

(iii) If $A_0$ is a bounded operator in $H$, then
$\cA$ is a bounded operator in $H_{\frac 12}\times H$ by Proposition~\ref{Ainv} and so $W(\cA)$ is bounded.

Vice versa, suppose that $A_0$, and hence \smash{$A_0^{\frac 12}$}, is unbounded. Then, since $H_1$ is a core for \smash{$A_0^{\frac 12}$},
there exists a sequence $(g_n)_{n\in \mathbb N}$ in $H_1$ with $\|g_n\|=\frac{1}{\sqrt{2}}$ such that $\la
A_0^{\frac{1}{2}}g_n,g_n\ra \to \infty$, $n\to \infty$. For
$n\in \mathbb N$, we set 
\[
f_n:= \left\{ \begin{array}{cl} \i \,A_0^{-\frac{1}{2}}g_n & \text{ if }\ \im\la Dg_n,g_n\ra =0, \\[1mm]
\i \ds \frac{\im \la Dg_n,g_n\ra}{\left|\im \la Dg_n,g_n\ra \right|}\,A_0^{-\frac{1}{2}}g_n & \text{ otherwise.}
\end{array} \right.
\]
Obviously, $\smash{\|f_n\|^2_{\hhalf}}\! + \|g_n\|^2 = 1$ and, by (A3), $(f_n, g_n)^{\top} \in H_1\times H_1 \subset \dom(\cA)$. From Remark~\ref{numericalrange} we deduce
\[
\im \left\la \cA \begin{bmatrix}f_n \\g_n\end{bmatrix}, \begin{bmatrix}f_n \\g_n\end{bmatrix}\right\ra_{\!H_{\frac{1}{2}}\times H\!}
= \left(\!-2 \la g_n,  A_0^{\frac{1}{2}} g_n\ra- \left|\im \la
Dg_n,g_n\ra\right|\right)\frac{\im \la Dg_n,g_n\ra}{\left|\im \la
Dg_n,g_n\ra \right|}
\]
if $\im\la Dg_n,g_n\ra \not=0$ and
\[ \im \left\la \cA \begin{bmatrix}f_n \\g_n\end{bmatrix}, \begin{bmatrix}f_n \\g_n\end{bmatrix}\right\ra_{\!H_{\frac{1}{2}}\times H\!}
= -2 \la g_n,  A_0^{\frac{1}{2}} g_n\ra\]
if $\im\la Dg_n,g_n\ra =0$,
which shows that the imaginary part of $W(\cA)$ is unbounded.

The last claim follows from Remark~\ref{numericalrange} if we use that $A_0$ bounded implies $D$ bounded and that in \eqref{WA} we can estimate
$| 2 \la A_0^{\frac 12} f, A_0^{\frac 12} g \ra | \le 2 \|A_0^{\frac 12}\| \|f\|_{H_{\frac 12}} \|g\| \le \|A_0^{\frac 12}\| (\|f\|_{H_{\frac 12}}^2\! + \|g\|^2 ) = \|A_0^{\frac 12}\|$,
and $|\im \la Dg,g \ra_{H_{-\frac 12},H_{\frac 12}}| \le \mathopen\|\Im D\| \|g\|^2 \le \mathopen\|\Im D\|$.
\end{proof}

The following example shows that the numerical range $W(\cA)$ may indeed fill the entire closed left half-plane.

\begin{example}
Let $H=\ell^2(\mathbb N)$,  $\mathbb N=\{1,2,\ldots\}$. The operator
\begin{eqnarray*}
&&\dom (A_0) := \bigl\{ (x_n)_{n\in\N} \in \ell^2(\mathbb N) \mid (nx_n)_{n\in\N} \in \ell^2(\mathbb N)\bigr\}, \\
&&A_0 (x_n)_{n\in\N}:= (nx_n)_{n\in\N}, \quad (x_n)_{n\in\N}\in \dom (A_0),
\end{eqnarray*}
satisfies {\rm (A1)} and $H_{\frac{1}{2}}=\left\{ (x_n)_{n\in\N}\in \ell^2(\mathbb N) \mid (\sqrt{n}\,x_n)_{n\in\N} \in \ell^2(\mathbb N)\right\}$.
Then the operator
\[
D(x_n)_{n\in\N} := \bigl((1+(-1)^n)nx_n\bigr)_{n\in\N}, \quad (x_n)_{n\in\N}\in H_{\frac{1}{2}}\,,
\]
satisfies {\rm (A2)} and {\rm (A3)}. As usual, we denote by $e_j := (\delta_{ij})_{i\in\N}$, $j\in \mathbb N$, the sequence of unit vectors in $\ell^2(\mathbb N)$.
Then, clearly $W(-D) = (-\infty,0] $ and hence $(-\infty,0] \subset W(\cA)$ by \eqref{WAincl}.
Moreover, for $n \in \mathbb N$,
\[
\|2^{-\frac 12} (2n+1)^{-\frac{1}{2}}\,e_{2n+1}\|_{\hhalf} = \frac{1}{\sqrt{2}}\,, \qquad
\mathopen\|\pm \i 2^{-\frac{1}{2}} \,e_{2n+1}\| = \frac{1}{\sqrt{2}}\,,
\]
and, by \eqref{WA} since $De_{2n+1} =0$,
\begin{eqnarray*}
&&\left\la \cA \begin{vek}2^{-\frac 12} (2n+1)^{-\frac{1}{2}}\,e_{2n+1}\\
\pm \i 2^{-\frac{1}{2}} \,e_{2n+1}
\end{vek} , \begin{vek} 2^{-\frac 12} (2n+1)^{-\frac{1}{2}}\, e_{2n+1}\\
\pm \i 2^{-\frac{1}{2}} \,e_{2n+1}
\end{vek}\right\ra_{H_{\frac{1}{2}}\times H}  \\
&& \ =
\frac{\pm \i}{\sqrt{2n+1}} \left\la \sqrt{2n+1}\,e_{2n+1},
\sqrt{2n+1}\,e_{2n+1}\right\ra =(\pm \i) \sqrt{2n+1}\ \longrightarrow \ \pm\i\infty, \quad n\to\infty.
\end{eqnarray*}
Altogether, the convexity of $W(\cA)$, see e.g.\ \cite[Theorem~\V.3.1]{K},
implies that $W(\cA)$ is the entire closed left half-plane, 
\[
  W(\cA) = \{ \lambda \in \C \mid \re \lambda \le 0 \}.
\]
\end{example}

\vspace{1mm}

Due to the spectral inclusion property \eqref{nr-specincl}, estimates for the numerical range yield estimates for the approximate point spectrum.
For the spectrum of $\cA$, we obtain the following.

\begin{cor}
\label{may4}
The spectrum of $\cA$ satisfies the following inclusions:
\begin{enumerate}
\item[{\rm (i)}]
$\sigma(\cA) \subset \{z \in \C \setminus \{0\} \mid \re z \le 0\}$;
\begingroup
\item[{\rm (ii)}]
\thickmuskip=3mu plus 2mu minus 1mu
if $\,\gamma < \infty$ and there is $\lambda_0 \in \rho(\cA)$ with $\re \lambda_0 < - \gamma$, then $\sigma(\cA) \subset \bigl\{z \in \C \setminus \{0\} \mid -\gamma \le \re z \le 0 \bigr\}$;
\endgroup
\item[{\rm (iii)}] if $\,\gamma < \infty$ and $\mu>0$, then 
$
\sigma(\cA) \subset \Bigl\{z \in \C \setminus \{0\} \mid - \gamma \le \re z \le 0, \ |\im z| \le \dfrac \gamma\mu + \mathopen\|\Im D\| \Bigr\}
$,
and the same inclusion holds if $\,\gamma < \infty$ and $\delta>0$.
\end{enumerate}
\end{cor}

\begin{proof}
\thickmuskip=5mu plus 3mu minus 1mu
By Proposition~\ref{Ainv}, we know $0\notin\sigma(\cA)$. Thus,
by Proposition~\ref{1lemma} and Lemma~\ref{0alemma}, in all claims it suffices to prove that $\sigma(\cA) \subset \overline{W(\cA)}$.
As $\overline{W(\cA)}$ is convex, the set $\C \setminus \overline{W(\cA)}$ consists of one or two components.
By \cite[Theorem~\V.3.2]{K}, 
if a component $\Omega$ of $\C \setminus \overline{W(\cA)}$ contains a point
$\lambda_0\in \rho(\cA)$, then $\Omega \subset \rho(\cA)$.
Since $0 \in \rho(\cA)$ by Proposition~\ref{Ainv} and $\rho(\cA)$ is open, we always have
$\{z \in \C \mid \re z > 0\} \cap \rho(\cA) \ne \emptyset$ and thus~(i) follows.
The assumption in (ii) ensures that also $\{z \in \C \mid \re z < -\gamma \} \cap \rho(\cA) \ne \emptyset$.
By Lemma~\ref{0alemma} the assumptions in (iii) guarantee that $\cA$ is bounded which implies that $\sigma(\cA) \subset \overline{W(\cA)}$ and hence the claim follows.
\end{proof}

\begin{rem}\label{spektrum2}
Corollary~\rmref{may4} provides an alternative proof for the fact that $\sigma(\cA)$ is contained in the closed left half-plane, see Remark~\rmref{spektrum}.
\end{rem}

If the operator $D$ has some sectoriality property, then the
numerical range of $\cA$ is contained in some parabolic region, as the following result shows.
We point out that the numerical range cannot lie in a sector with corner $0$:
recall from Proposition~\ref{1lemma}~(i) that $0 \in W(A)$.
Thus $0$ being a corner of $W(A)$ would imply $0\in \sigma(\cA)$, cf.\ \cite{LMT},
a contradiction to Proposition~\ref{Ainv}.

\begin{prop}
\label{SpekInclusion}
Assume there exists $k\geq0$ such that
\begin{equation}
\label{Dietendorf}
|\im \langle Dz,z \rangle_{H_{\frac{1}{2}}\times H_{-\frac{1}{2}}}|\leq k \,\re \langle Dz,z
\rangle_{H_{\frac{1}{2}}\times H_{-\frac{1}{2}}}, \quad z\in H_{\frac{1}{2}}\,.
\vspace{-2.5mm}
\end{equation}
If $\delta>0$, then \vsqueeze{2mm}
\begin{align}
\label{verylast}
\sigma(\cA) \subset \overline{W(\cA)} \subset \left\{\lambda \in \C \mid -\gamma\le \re \lambda \leq 0, \
|\im\lambda| \leq k |\re \lambda |  + 2 \sqrt{\frac 1\delta|\re \lambda |}\,
 \right\}.
\end{align}
\end{prop}

\begin{proof}
Proposition~\ref{1lemma}~(ii) implies  $W(\cA) \subset \bigl\{\lambda \in \C \mid -\gamma\le \re  \lambda \leq 0 \bigr\}$.
By Remark~\ref{numericalrange}, we have $\lambda \in W (\cA)$ if and only if there exists
$\rule{0pt}{2.35ex}(f, g)^{\top} \in \dom(\cA)$ with $\smash{\|f\|^2_{\hhalf}}\! + \|g\|^2= 1$
such that
\begin{align}
\label{Senat0}
\re \lambda &= -\re \la Dg,g\ra_{H_{-\frac{1}{2}}\times H_{\frac{1}{2}}}, \\
\label{Senat}
\im \lambda &= \smash{-2\, \im\la f, g\ra_{\hhalf}\!
               - \im \la Dg,g\ra_{H_{-\frac{1}{2}}\times H_{\frac{1}{2}}}.}
               \vphantom{\hhalf}
\end{align}
If $\delta\!>\!0$, then $\|g\|^2_{\hhalf} \!\!\le\! \frac 1\delta\re \la Dg,g\ra_{H_{-\frac{1}{2}}\times
  H_{\frac{1}{2}}}$. Using this estimate, $\|f\|_{\hhalf}\!\!\leq\! 1$, \eqref{Senat0} and \eqref{Senat}, we~find
\begin{align*}
   |\im \lambda |
   &\leq 2\|f\|_{\hhalf} \|g\|_{\hhalf} + |\im \la Dg,g\ra_{H_{-\frac{1}{2}}\times  H_{\frac{1}{2}}}|\\
   &\leq 2\sqrt{\frac 1\delta \re \la Dg,g\ra_{H_{-\frac{1}{2}}\times H_{\frac{1}{2}}}} + k  \, \re \la Dg,g\ra_{H_{-\frac{1}{2}}\times H_{\frac{1}{2}}}\\
   &=   2\sqrt{\frac 1\delta |\re \lambda |} + k   |\re \lambda |,
  \end{align*}
which proves the inclusion for  $\overline{W(\cA)}$ in \eqref{verylast}. This and
the convexity of the set $\overline{W(\cA)}$ ensures that
the complement $\C\setminus \overline{W(\cA)}$ has only one component,
in both cases $\gamma=\infty$ and $\gamma<\infty$.
Now the inclusion $\sigma(\cA) \subset \overline{W(\cA)}$ in \eqref{verylast} follows from \cite[Theorem~\V.3.2]{K} in the same way as the inclusion in Corollary~\ref{may4} since we know $0\in \rho(\cA)$.
\end{proof}

\section{The quadratic numerical range (QNR) of $\cA$}\label{section4}

In this section we establish new spectral enclosures for the operator $\cA$ in \eqref{AA1} by means of the so-called quadratic numerical range.
The latter is defined for operators in a product Hilbert space ${\mathcal H}_1 \times {\mathcal H}_2$ that admit a matrix representation with respect to some decomposition of the space,
i.e.\ that have a domain of the form
$\mathcal D_1\times \mathcal D_2$ with dense subspaces $\mathcal D_i$ of ${\mathcal H}_i$, $i=1,2$.

In general, such a decomposition of the domain of the operator $\cA$ in \eqref{AA1} requires stronger assumptions on $D$; e.g.\ if $D$ maps $H_{\frac 12}$ even into~$H$, then $\dom(\cA) = H_1 \times H_{\frac 12}$.
Under the weaker assumptions (A2), (A3), $H_1\!\times\! H_1 \subset \dom(\cA)$ is a core of $\cA$ by Proposition~\ref{Ainv} and so the quadratic numerical range of the restriction
$\cA|_{H_1\times H_1}$ is defined as follows, see \cite[Definition 2.5.1]{T}.

\begin{defn}\label{QNR}
For $(f, g)^{\top} \in H_1\times H_1 \subset \dom(\cA)$, $f,g \ne 0$, let
\begin{equation*}\label{Afg2}
 \cA_{f,g} := \begin{mat} 0 &
\dfrac{\la g,f\ra_{\hhalf}}{\|f\|_{\hhalf} \|g\|} \\[2ex]
-\dfrac{\la A_0f,g\ra }{\|f\|_{\hhalf}
 \|g\|}& -\dfrac{\la Dg,g\ra}{\|g\|^2} \end{mat}\in M_2(\C).
\end{equation*}
The set of all eigenvalues of all these $2\!\times\!2$ matrices $\cA_{f,g}$,
\begin{equation*}
 W^2(\cA|_{H_1\times H_1}) :=
 \bigcup_{\substack{(f,g)^{\top} \in\mkern2mu H_1\times H_1, \\[0.2ex]
 f,g \ne 0}}
  \sigma_p(\cA_{f,g}) =
 \bigcup_{\substack{(f,g)^{\top} \in\mkern2mu H_1\times H_1, \\[0.2ex]
 \|f\|_{\hhalf} \mkern-2mu = \|g\|= 1}}
  \sigma_p(\cA_{f,g}),
\end{equation*}
is called the {\em quadratic numerical range} of the operator matrix $\cA|_{H_1\times H_1}$ in $H_{\frac 12} \times H$.
\end{defn}

\begin{rem}\label{remark}
The following equivalent description of $W^2(\cA|_{H_1\times H_1})$ is useful, see {\rm \cite[Proposition 1.1.3]{T}}.
For $(f, g)^{\top} \in  H_1\times H_1$ with $f,g \ne 0$, set
\begin{align}
\Delta (f,g;\lambda)
&:=\|f\|_{\hhalf}^2\|g\|^2\det(\cA_{f,g}\!-\!\lambda) 
= \|f\|_{\hhalf}^2\|g\|^2
\left( \lambda^2 +\lambda \frac{\langle Dg,g\rangle}{\|g\|^2} +
\frac{|\langle f,g\rangle_{\hhalf}|^2}{\|f\|^2_{\hhalf} \|g\|^2}\right).\label{Afg3}
\vspace{-2mm}
\end{align}
Then
\begin{equation}\label{QNR2}
 W^2(\cA|_{H_1\times H_1})
 = \bigl\{\lambda \in \C \mid \exists\, (f,g)^{\!\top} \! \in
 H_1\!\times\! H_1,\ f,g \ne 0:  \Delta(f,g;\lambda) =0\bigr\}.
\end{equation}
\end{rem}

\smallskip

The quadratic numerical range is either connected or consists of two components;
thus it is in general not convex, and even its components need not be so (see e.g.\ \cite{LMMT}, \cite[p.~4/5]{T}).

An important property of the quadratic numerical range is that it is always contained in the
numerical range. Together with Proposition~\ref{1lemma}, we obtain: 

\begin{prop}
\label{qnrincl}
\[
W^2(\cA|_{H_1\times H_1}) \subset W(\cA|_{H_1\times H_1}) \subset W(\cA) \subset \{ z\in \C \mid -\gamma \le \re z \le 0\}.
\]
\end{prop}

\begin{proof}
The first inclusion was proved in \cite[Theorem~2.5.3]{T}, the second one is obvious, and the third one was shown in Proposition~\ref{1lemma}~(ii).
\end{proof}

In general, the quadratic numerical range may be considerably smaller than the numerical range.
The next proposition shows that the extreme points of their real parts are the same.

\begin{prop}
\label{vonC}
If\/ $\dim H > 1$, then
\begin{equation}
\label{diagincl}
W(-D) \cup \{0\}
\subset W^2(\cA|_{H_1\times H_1}) \cap W(\cA)
\end{equation}
and hence
\begin{equation}
\label{oberwolfach-qnr}
\inf \bigl( \Re W^2(\cA|_{H_1\times H_1}) \bigr) = -\gamma, \quad
\max \bigl( \Re W^2(\cA|_{H_1\times H_1}) \bigr) = 0.
\end{equation}
\end{prop}

\begin{proof}
Since $\dim H_{\frac 12}, \dim H > 1$, the numerical ranges of the diagonal elements of $\cA|_{H_1\times H_1}$,
i.e.\ of the zero operator $0$ in $H_{\frac 12}$ and of $D:H\to H$ with
 $\dom (D) = H_1$,
 are contained in $W^2(\cA|_{H_1\times H_1})$ by \cite[Theorem~2.5.4]{T}. This together with Proposition~\ref{1lemma}~(i) proves \eqref{diagincl}.

The claims in \eqref{oberwolfach-qnr} follow from \eqref{diagincl}, Proposition~\ref{qnrincl} and Proposition~\ref{1lemma}~(ii).
\end{proof}

\section{The spectral inclusion property of the QNR}\label{section5}

In this section we establish the spectral inclusion property of $W^2(\cA|_{H_1\times H_1})$ under our standard assumptions (A1)--(A3).
To obtain inclusions for the spectrum of $\cA$, we use that
$\cA= \overline{\cA|_{H_1\times H_1}}$ by Proposition~\ref{Ainv} and hence, see e.g.\ \cite[Lemma 2.5.16]{T},
\begin{equation}
\label{sigma-core}
  \sigma_p(\cA) \subset \sigma_{ap}(\cA|_{H_1\times H_1}), \quad  \sigma_{ap}(\cA) = \sigma_{ap}(\cA|_{H_1\times H_1}).
\end{equation}

\begin{thm}
\label{Karl} We have
\begin{alignat}{2}
\label{sig-incl1}
\sigma_{p}(\cA|_{H_1\times H_1}) &\subset W^2(\cA|_{H_1\times H_1}), & \quad
\sigma_{ap}(\cA|_{H_1\times H_1}) &\subset \overline{W^2(\cA|_{H_1\times H_1})}, \\[-7mm]
\intertext{and \vspace{-3mm} hence}
\label{sig-incl2}
\sigma_{p}(\cA) &\subset \overline{W^2(\cA|_{H_1\times H_1})}, &\quad
\sigma_{ap}(\cA) &\subset \overline{W^2(\cA|_{H_1\times H_1})}.
\end{alignat}
\end{thm}

\begin{proof}
It suffices to prove the inclusions \eqref{sig-incl1}; the inclusions \eqref{sig-incl2} follow from \eqref{sig-incl1} by means of~\eqref{sigma-core}.

The inclusion of the point spectrum in \eqref{sig-incl1} was proved in \cite[Theorem~2.5.9]{T}.
To prove the inclusion of the approximate point spectrum, let $\lambda \in \sigma_{ap}(\cA|_{H_1\times H_1})=\sigma_{ap}(\cA)$.
Then, by Proposition~\ref{Ainv}, $\lambda \ne 0$, $\re \lambda \leq 0$, and there exists a sequence
$((f_n,g_n)^\top)_{n\in\N}$ in  $H_1\times H_1$
with
\begin{equation*}\label{folge1}
\left\|\begin{bmatrix} f_n\\g_n \end{bmatrix} 
\right\|_{H_{\frac{1}{2}} \times H}
 =1, \quad 
\quad \lim_{n \to \infty}
\left\|(\cA-\lambda) \begin{bmatrix} f_n\\g_n \end{bmatrix} 
\right\|_{H_{\frac{1}{2}} \times H} =0.
\end{equation*}
Then we have \vsqueeze{2mm}
\begin{align}
\label{folge0}
\|f_n\|^2_{\hhalf} + \|g_n\|^2 &=1 \\[-7mm]
\intertext{\vspace{-4mm}and}
\label{folge2}
\|g_n-\lambda f_n\|_{\hhalf} &\to 0, \\
\label{folge3}
\|A_0 f_n + Dg_n + \lambda g_n\| &\to 0, \quad n \to \infty.
\end{align}
Without loss of generality, we may assume that
\[
  \smash{ a  := \lim_{n\to\infty} \|f_n\|^2_{\hhalf}}
\]
exists. Then $ b  := \lim_{n\to\infty} \|g_n\|^2 = 1- a $, by~\eqref{folge0}.
If $ a =0$, then \eqref{folge2} and \eqref{a_0} imply that $ b =0$,
a contradiction to $ b =1- a $.
Hence we have $ a >0$.

Now we consider the sequence of polynomials
\begin{equation}\label{DickDelta}
\Delta( f_n, g_n;z) = \det\!
\begin{mat}
\!\! -z \la  f_n, f_n\ra_{\hhalf} & \la g_n, f_n\ra_{\hhalf} \!\!  \\[2mm]
\!\! - \la A_0  f_n, g_n\ra  & -\la D g_n, g_n\ra - z \la  g_n,  g_n\ra \!\!
\end{mat}, \quad z\in\C, \ n\in\N.
\end{equation}
By \eqref{folge2} we obtain
\begin{equation}\label{Oberhof2}
  \lim_{n\to\infty} \la g_n, f_n\ra_{\hhalf} =
  \lim_{n\to\infty} \la \lambda f_n, f_n\ra_{\hhalf} = \lambda a .
\end{equation}
It follows that $\lim_{n\to\infty} \la A_0 f_n, g_n \ra = \overline{\lambda} a $.
Note that $(g_n)_{n\in\N}$ is bounded in $H$ by~\eqref{folge0}.
Thus, using \eqref{folge3} and the definitions of $ a , b $, we deduce that
\begin{equation}\label{Oberhof5}
  \lim_{n\to\infty} \la Dg_n, g_n \ra =
  -\lim_{n\to\infty} \la A_0 f_n + \lambda g_n, g_n \ra =
  -\overline{\lambda} a  - \lambda b .
\end{equation}

Then, by \eqref{DickDelta}, \eqref{Oberhof2} and \eqref{Oberhof5}, it follows that
\[
\Delta( f_n, g_n;z) \to \det\!
\begin{mat}
\!\! -z  a  & \lambda a  \!\!  \\[0.5mm]
\!\! -\overline{\lambda} a  & \overline{\lambda} a  + \lambda b  - z b  \!\!
\end{mat}
=: \Delta(z), \quad n\to\infty,
\]
uniformly for $z$ in com\-pact subsets of $\C$. It is easy to see that $\Delta(\lambda) = 0$
and $\Delta \not\equiv 0$ since $\lambda a  \ne 0$. Hence, by Hur\-witz' theorem
(see e.g.\ \cite[Theorem~VII.2.5]{Con2}), for every $\varepsilon > 0$
there exists $N\mkern-1.5mu \in\N$ with the property that, for $n\ge N$,
the quadratic polynomial $\Delta(f_n, g_n;z)$ has a zero
$z_{n,1}\in\C$ with $|z_{n,1} - \lambda| < \varepsilon$.
Since $z_{n,1} \in  W^2(\cA|_{H_1\times H_1})$, it follows that
$\lambda \in \overline{W^2(\cA|_{H_1\times H_1})}$.
\end{proof}

\begin{prop}\label{spezial}
If, in addition to the assumptions {\rm (A2)}, {\rm (A3)}, the operator $D$ maps the space $H_{\frac{1}{2}}$ into $H$, then
\vsqueeze{2mm}
\begin{equation}
\label{ZZ2}
\sigma_p(\cA) = \sigma_p(\cA|_{H_1\times H_1})  \subset W^2(\cA|_{H_1\times H_1}).
\end{equation}
\end{prop}

\begin{proof}
By Theorem~\ref{Karl}, we only have to prove the first identity.
If $D$ maps $H_{\frac{1}{2}}$ into $H$, we have $\dom(\cA) = H_1\times H_{\frac{1}{2}}$.
Since an eigenvector $(f,g)^{\top} \in \dom(\cA)= H_1\times H_{\frac{1}{2}}$ of $\cA$ at an eigenvalue 
$\lambda \in \sigma_p(\cA)$ satisfies
\vsqueeze{1mm}
\[
-\lambda f + g =0,
\]
we see that also $g\in H_1$, and $\sigma_p(\cA) = \sigma_p(\cA|_{H_1\times H_1})$ follows.
\end{proof}

\begin{rem}
The stronger assumption in Proposition~\rmref{spezial} is satisfied if e.g.\ $D=A_0^\theta$ for some $\theta \in (-\infty, \frac{1}{2}]$.
\end{rem}

\smallskip

The following inclusion of the spectrum is immediate from
Theorem~\ref{Karl}.

\begin{thm}
\label{Spektrum}
If a component $\Omega$ of $\C \setminus \overline{W^2(\cA|_{H_1\times
H_1}})$ contains a point $\lambda_0\in \rho(\cA)$, then $\Omega
\subset \rho(\cA)$; in particular, if every component of $\mathbb
C\setminus \overline{W^2(\cA|_{H_1\times H_1}})$ contains a point $\lambda_0\in
\rho(\cA)$, then
\[
\sigma(\cA) \subset \overline{W^2(\cA|_{H_1\times H_1}}).
\]
\end{thm}

\begin{proof}
The claim follows from Theorem~\ref{Karl} and the fact that the boundary of
the spectrum $\sigma(\cA)$ belongs to $\sigma_{ap}(\cA)$,
see e.g.\ \cite[IV \S 1.10]{EN}.
Alternatively, it follows from Theorem~\ref{Karl} and the fact that the mapping $\lambda \mapsto \dim \ran (\cA-\lambda)^\perp$ is locally constant, see \cite[Theorem~\V.3.2]{K}.
\end{proof}

\section{Uniformly accretive and sectorial damping:\\
estimates for QNR and spectrum
}
\label{section6}

In this section and the next we show how special properties of the damping operator $D$ such as uniform accretivity and sectoriality are reflected in the quadratic numerical range $W^2(\cA|_{H_1\times H_1})$.
As a result we obtain new bounds on the spectrum of~$\cA$ which improve the bounds by the numerical range,
see Proposition~\ref{SpekInclusion}, considerably.

In particular, we show that the spectrum may have a gap around $\re \lambda = -\frac\beta 2$ if $\delta>0$;
in this case, $D$ is uniformly accretive with $\inf\,(\Re W(D)) = \beta \ge a_0^2 \delta > 0$, see \eqref{bgdm} and~\eqref{ineq}.
Note that the spectral free strip has to lie between $-\beta$ and $0$ since $W(-D) \subset W^2(\cA|_{H_1\times H_1})$ by~\eqref{diagincl}.
We also show that, unlike the numerical range, $W^2(\cA|_{H_1\times H_1})$ may lie in a sector with corner $0$ even though 
$0\,\notin\sigma(\cA)$ since the zero operator on the diagonal of $\cA$ has $0$ in its spectrum, cf.\ \cite[Theorem~3.1]{LMT}.

\begin{thm}
\label{hendrik}
Suppose that $\delta>0$ and hence $\beta>0$, so that $D$ is uniformly accretive. Then,
\[
 \sigma(\cA) \subset 
 \bigl\{ \lambda \in \C \mid \re\lambda < 0, \
 |\re\lambda| \notin I_0, \
 |\im \lambda| \le h_0(|\re\lambda|) \bigr\}
\]
where $I_0$ is a $($possibly empty$)$ interval centred at \smash{$\frac \beta 2$}, given by
\vsqueeze{1mm}
\begin{equation}
\label{I0}
 I_0  := \begin{cases}
                \hspace{2.7cm}\emptyset & \mbox{ if } \ \beta\delta \le 4,\\
								\biggl( \frac \beta 2 \Bigl( 1 - \sqrt{1 \!-\! \frac 4{\beta\delta}}\, \Bigr),
								       \frac \beta 2 \Bigl( 1 + \sqrt{1 \!-\! \frac 4{\beta\delta}}\, \Bigr) \biggr)  &
								\mbox{ if } \ \beta\delta > 4,			
                \end{cases}
\end{equation}
and
\[
 h_0(t) := \begin{cases}
 \displaystyle \sqrt{\frac \beta\delta \, \frac{t}{\beta-t}-t^2}, & \ 0 \le t < \beta, \
 t \notin I_0,
 \\
 \hspace{10mm} \infty, & \ \beta \le t < \infty;
\end{cases}
\]
in particular, if $\,\beta\delta > 4$, then $\cA$ has a spectral free strip around $\re \lambda = - \frac \beta 2$,
\[
 \sigma(\cA) \cap \bigl\{\lambda \in \C \mid
 |\re\lambda| \in I_0
\bigr\} = \emptyset.
\]
If $\gamma < \infty$ and there is a $\lambda_0\in \rho(\cA)$ with $\re \lambda_0 < -\gamma$, then
$\sigma(\cA) \cap \{\lambda\in\C \mid \re\lambda<-\gamma\}= \emptyset$.
\end{thm}

\begin{proof}
If we show that $W^2(\cA|_{H_1\times H_1}) \setminus \{0\}$ satisfies the asserted inclusion,
then so does $\sigma(\cA)$ due to Theorem~\ref{Spektrum}, the fact that $0\in \rho(\cA)$ by Proposition~\ref{Ainv} and
that $h_0$ is bounded on the subinterval in $[0,\frac \beta 2]$ where it is defined with $h_0(0)=0$.

Since $\Re W^2(\cA|_{H_1\times H_1}) \le 0$ it suffices to consider $\lambda\in W^2(\cA|_{H_1\times H_1})$ with
$-\beta< \re \lambda \le 0$. By Definition~\ref{QNR},
there exists $(f,g)^{\top}  \in  H_1\times H_1$, $\|f\|_{\hhalf} = \|g\| = 1$,~with
\begin{equation}
\label{Simon3}
0 = \det(\cA_{f,g}-\lambda)
= \lambda \bigl( \lambda + \la Dg,g \ra \bigr) + |\la f,g\ra_{\hhalf}|^2.
\end{equation}
Together with $| \la f,g\ra_{\hhalf}|^2 \le \|g\|_{\hhalf}^2 \le \dfrac{\re \la Dg,g \ra}{\delta}$, this implies that
\[
 \re \la Dg,g \ra = -  |\la f,g\ra_{\hhalf}|^2 \,\re \frac 1 \lambda - \re \lambda \le
 \Bigl( \dfrac{\re \la Dg,g \ra}{\delta} \frac 1{|\lambda|^2} + 1 \Bigr) |\re \lambda|
\]
and hence
\[
  \frac{1}{|\re\lambda|} \le \frac{1}{\delta|\lambda|^2} + \frac{1}{\re\la Dg,g\ra}\,.
\]
Using this estimate and $\re \la Dg,g \ra \ge \beta > |\re\lambda| >0$, we obtain
\[
|\re\lambda|^2+|\im\lambda|^2 = |\lambda|^2
\le \frac 1 \delta \Bigl( \frac 1{|\re\lambda|} - \frac 1{\re\la Dg,g\ra}\Bigr)^{\!-1} \!
\le \frac 1 \delta \Bigl( \frac 1{|\re\lambda|} - \frac 1{\beta}\Bigr)^{\!-1} \!
= \frac 1 \delta \frac{\beta |\re\lambda|}{\beta- |\re\lambda|}\,,
\]
which proves the claimed spectral inclusion. 
Note that, if $\beta\delta>4$, then estimating the left hand side above further by
$|\re\lambda|^2 \le |\re\lambda|^2+|\im\lambda|^2$ yields that $|\re\lambda|$ must satisfy the inequality
$|\re\lambda|(\beta- |\re\lambda|) \le \frac \beta \delta$ or, equivalently, $\bigl| |\re\lambda| - \frac \beta 2\bigr|^2 \ge
\bigl(\frac \beta 2\bigr)^2 \bigl( 1-\frac 4 {\beta\delta} \bigr)
>0$.

The last assertion follows from Proposition~\rmref{qnrincl}.
\end{proof}

\begin{thm}\label{supernett}
Assume there exists $k\geq 0$ such that
\begin{equation}\label{Simon1}
|\im \langle Dz,z \rangle|\leq k \,\re \langle Dz,z \rangle,
\quad  z\in H_1.
\end{equation}
\begin{enumerate}
\item[{\rm (i)}] If $\,\beta>0$, then
\begin{equation}\label{Simon2a}
\sigma(\cA)
\subset
\bigl\{\lambda \in \C \mid \re \lambda < 0,
\
|\im \lambda| \le h_{\rm i}(|\re\lambda|)
\bigr\}
\end{equation}
where $h_{\rm i}:[0,\infty) \to [0,\infty]$ is given by \vsqueeze{2mm}
\begin{equation}
h_{\rm i}(t) :=
\begin{cases}
\frac 1{ 1- \frac 2\beta t} k t, & 0 \le t < \dfrac \beta 2, \\
\hspace{3mm} \infty, & \dfrac \beta 2 \le t < \infty;
\end{cases}
\vspace{-1mm}
\end{equation}
if $\gamma < \infty$ and there is a $\lambda_0\in \rho(\cA)$ with $\re \lambda_0 < -\gamma$, then
\[
\sigma(\cA)
\subset
\bigl\{\lambda \in \C \mid -\gamma \le \re \lambda < 0,
\
|\im \lambda| \le h_{\rm i}(|\re\lambda|)
\bigr\}.
\]
\item[{\rm (ii)}] If $\,\mu > 0$, then
\begin{equation}
\label{Simon2}
\sigma(\cA)
\subset
\left\{\lambda \in \C \mid -\gamma \le \re \lambda < 0, \
|\im \lambda| \le h_{\rm ii}(|\re\lambda|)
\right\}
\end{equation}
where $h_{\rm ii}:[0,\infty) \to [0,\infty)$ is given by
\begin{equation}
\label{kmu}
h_{\rm ii}(t) := k_\mu t, \quad k_\mu^2 :=
\frac 2{\mu^2} + \frac{k^2\!-\!1}{2}
+ \smash{\sqrt{\Bigl( \frac 2{\mu^2} + \frac{k^2\!-\!1}{2}\Bigr)^{\!2}\!+k^2}\,,}
\end{equation}
with $k_\mu \in [0,\infty)$ satisfying $k \le k_\mu \le \sqrt{k^2+\frac 4{\mu^2}}$.
\item[{\rm (iii)}] If $\,\delta>0$, then \vsqueeze{1mm}
\begin{equation}\label{Simon2b}
\sigma(\cA) 
\subset \bigl\{\lambda \in \C  \mid -\gamma \le \re \lambda < 0, \
|\im \lambda| \le h_{\rm iii}(|\re\lambda|)
\bigr\}
\end{equation}
where $h_{\rm iii}:[0,\infty) \to [0,\infty)$ is defined by $h_{\rm iii}(t)$ being the largest non-negative
solution $y$ 
of 
\begin{equation}
\label{cubic}
(y^2+t^2)(y-kt) = \frac 2\delta ty,
\end{equation}
which satisfies the estimates
\begin{equation}
\label{veryverylast}
kt \le  h_{\rm iii}(t) \le \min \biggl\{ kt+\frac 1\delta, \frac {kt}2 + \sqrt{\Bigl( \frac{kt}2\Bigr)^{\!2}\! + \mkern-1mu \frac{2t}\delta} \,\biggr\}
\le kt + \min \biggl\{ \frac 1\delta , \sqrt{\frac{2t}{\delta}} \mkern1.5mu\biggr\},
\quad t\!\in\! [0,\infty).
\end{equation}
\end{enumerate}
\end{thm}

\begin{rem}
\label{kmu-rem}
(a)
If $k>0$, then the function $\mu \mapsto k_\mu$ is strictly decreasing on $(0,\infty)$ from a pole at $\mu=0$ to $\lim_{\mu\to\infty} k_\mu = k$;
for $k=0$, it is strictly decreasing on $(0,2)$ and equal to $0$ for $\mu \ge 2$,
\begin{equation}
\label{kmuk0}
 k_\mu^2 = \frac 2{\mu^2}-\frac 12 + \Bigl| \frac 2{\mu^2}-\frac 12 \Bigr|
         = \begin{cases}
           \frac 4{\mu^2} -1, & 0 < \mu < 2,\\
           \hspace{4.5mm} 0, & \hspace{7mm} \mu \ge 2,
          \end{cases}
\qquad \mbox{if } k=0.
\end{equation}
Note that, in general, $k_\mu=k$ if and only if $k=0$ and $\mu \ge 2$.

(b) The spectral enclosure \eqref{Simon2b} by the quadratic numerical range in  
Theorem~\ref{supernett}~(iii) is better than the one by the numerical range 
in Proposition~\ref{SpekInclusion}; %
indeed, the term $\sqrt{\rule[-0.3ex]{-0.1ex}{2.25ex}\smash{\frac{2t}{\delta}}}$ in the last upper bound for $h_{\rm iii}$ in \eqref{veryverylast} is better than the corresponding term there by a factor of $\sqrt2$.
\end{rem}

\begin{proof}[Proof of Theorem~\rmref{supernett}.]
\renewcommand{\hhalf}{H_{1/2}}
If we show that $W^2(\cA|_{H_1\times H_1}) \setminus \{0\}$ satisfies the asserted inclusions,
then so does $\sigma(\cA)$ due to Theorem~\ref{Spektrum}, the fact that $0\in \rho(\cA)$ by Proposition~\ref{Ainv}, and $h_k(0)=0$ for $k\in\{{\rm i}, {\rm ii}, {\rm iii}\}$.

Let $\lambda\in W^2(\cA|_{H_1\times H_1}) \setminus \{0\}$. Proposition~\ref{qnrincl} implies that
$-\gamma \le \re \lambda \leq 0$, so we only have to show the estimates for $\im\lambda$.
Further, we can assume that $\im\lambda \ne 0$ since
all enclosing sets contain $\{t\in\R \mid -\gamma\le t<0\}$.
By Definition~\ref{QNR}, there exists $(f,g)^{\top} \! \in  H_1\mkern-1mu\times\mkern-1mu H_1$ with $\|f\|_{\hhalf}\mkern-2mu = \|g\| = 1$
such that \eqref{Simon3} holds.
Dividing by $\lambda$ and taking real and imaginary parts, we~obtain
\begin{align}
\label{hendrik1}
\re \la Dg,g \ra &= \Biggl( \frac{|\la f,g\ra_{\hhalf}|^2}{|\lambda|^2}+1 \Biggr) |\re\lambda|, \\
\label{hendrik2}
\im \la Dg,g \ra &= \Biggl( \frac{|\la f,g\ra_{\hhalf}|^2}{|\lambda|^2}-1 \Biggr) \,\im\lambda.
\end{align}
Since in all cases $\beta >0$ and hence $\re \la Dg,g \ra \ne 0$, \eqref{hendrik1}
implies $\re\lambda \ne 0$ and hence we conclude
\begin{align}
\label{chr1}
 & \frac{\im\la Dg,g \ra}{\im\lambda}  = \frac{\re\la Dg,g \ra}{|\re\lambda|} -2  = \frac{\re\la Dg,g \ra-2\,|\re\lambda|}{|\re\lambda|}\,, \\
\label{chr2}
 & \frac{\re\la Dg,g \ra}{|\re\lambda|} + \frac{\im\la Dg,g \ra}{\im\lambda} =  2 \frac{|\la f,g\ra_{\hhalf}|^2}{|\lambda|^2}\,.
\end{align}

{\rm (i)}
Assume that $|\re\lambda|< \frac \beta 2$. By \eqref{hendrik1} and the definition of $\beta$ in \eqref{bgdm}, we have
\[
  \frac{|\la f,g\ra_{\hhalf}|^2}{|\lambda|^2}+1 = \frac{\re\la Dg,g \ra}{|\re\lambda|} \ge \frac \beta{|\re\lambda|} \ (\,>2).
\]
Then, from \eqref{hendrik1}, \eqref{hendrik2}, \eqref{Simon1} and the above estimate it follows that
\[
  \frac{|\im\lambda|}{|\re\lambda|} = \frac{|\im\la Dg,g \ra|}{|\re \la Dg,g \ra|} \,
	 \frac{\frac{|\la f,g\ra_{\hhalf}|^2}{|\lambda|^2}+1}
	 {\bigg|\frac{|\la f,g\ra_{\hhalf}|^2}{|\lambda|^2}-1\bigg|}
   	\le k\, \frac 1{1 - 2 \biggl( \frac{|\la f,g\ra_{\hhalf}|^2}{|\lambda|^2}+1\biggr)^{\!\raisebox{-2pt}{$\scriptstyle-1$}}}
	\le k\, \frac 1{1 - \frac 2\beta |\re\lambda|}\,.
\]

{\rm (ii)}
By~\eqref{hendrik1} and~\eqref{hendrik2} we obtain
\[
    \left( \frac{\re \la Dg,g \ra}{\re\lambda} \right)^{\!2}
  - \left( \frac{\im \la Dg,g \ra}{\im\lambda} \right)^{\!2}
  = 4\, \frac{|\la f,g\ra_{\hhalf}|^2}{|\lambda|^2}\,.
\]
Multiplying this identity by $\dfrac{|\lambda|^2}{(\re \la Dg,g \ra)^2}$, we infer that
\[
    \frac{|\lambda|^2}{|\re\lambda|^2}
  - \frac{|\lambda|^2}{|\im\lambda|^2}
    \left( \frac{\im \la Dg,g \ra}{\re \la Dg,g \ra} \right)^{\!2}
  = 4 \left( \frac{|\la f,g\ra_{\hhalf}|}{\re \la Dg,g \ra} \right)^{\!2}.
\]
Using $\|f\|_{\hhalf} \mkern-1.5mu = \|g\| = 1$ and the definition of $\mu$ in \eqref{bgdm}, we estimate
$|\la f,g\ra_{\hhalf}| \le \|g\|_{\hhalf} \mkern-1.5mu \le \frac1\mu \re \la Dg,g \ra$.
Thus from the sectoriality of $D$, i.e.\ from \eqref{Simon1}, it follows that
\[
  1 - k^2 + \frac{|\im\lambda|^2}{|\re\lambda|^2}
          - \frac{|\re\lambda|^2}{|\im\lambda|^2} k^2
  = \frac{|\lambda|^2}{|\re\lambda|^2} - \frac{|\lambda|^2}{|\im\lambda|^2} k^2
  \le \frac{4}{\mu^2}\,.
\]
Hence
\[
  \frac{|\im\lambda|^2}{|\re\lambda|^2} \left( \frac{|\im\lambda|^2}{|\re\lambda|^2} + 1 - k^2 - \frac 4{\mu^2} \right) - k^2 \le 0.
\]
The latter is a quadratic inequality for $\frac{|\im\lambda|^2}{|\re\lambda|^2}$.
If we note that $ 1 - k^2 - \frac 4{\mu^2} = - 2 \bigl( \frac 2{\mu^2} + \frac{k^2-1}{2} \bigr)$, we see that this inequality
is satisfied if and only if $\smash[t]{\frac{|\im\lambda|^2}{|\re\lambda|^2}} \le k_\mu^2$,
due to the definition of $k_\mu^2$.

The inequalities for $k_\mu$ are not difficult to check: for the lower bound we note that $k_\mu^2$ is strictly decreasing in $\mu$ and $\lim_{\mu\to\infty} k_\mu^2 = k^2$; for the upper bound we use the inequality $\bigl( \frac 2{\mu^2} + \frac{k^2-1}{2}\bigr)^2\!+k^2
\le \bigl( \frac 2{\mu^2} + \frac{k^2+1}{2}\bigr)^2$.

{\rm (iii)}
Multiplying \eqref{chr2} by $\dfrac{|\re\lambda| \,\im\lambda \,|\lambda|^2}{\re \la Dg,g \ra}$, we conclude that \vsqueeze{1mm}
\[
|\lambda|^2 \Bigl( \im \lambda + \frac{\im\la Dg,g \ra}{\re\la Dg,g \ra} |\re\lambda|  \Bigr) = 2 \frac{|\la f,g\ra_{\hhalf}|^2}{\re\la Dg,g \ra}|\re\lambda| \,\im\lambda.
\]
From the sectoriality of $D$, i.e.\ from \eqref{Simon1}, the inequality $\smash{|\la f,g\ra_{\hhalf}|^2} \le \smash{\|g\|_{\hhalf}^2}$,
and the definition of $\delta$ in \eqref{bgdm}, it follows that
\[
  |\lambda|^2 \bigl( |\im \lambda| - k|\re\lambda| \bigr)
  \le \smash[t]{\frac 2 \delta}  |\re\lambda| |\im\lambda|,
\]
which is satisfied if and only if $|\im\lambda|\le h_{\rm iii}(|\re\lambda|)$ by definition of $h_{\rm iii}$.

The three upper bounds for $h_{\rm iii}$ in \eqref{veryverylast} are not difficult to check:
for the first bound in the first inequality we use the estimate
$2ty \le t^2+y^2$ on the right hand side of \eqref{cubic},
while for the second bound in the first inequality we use
$y^2 \le (y^2+t^2)$ on the left hand side of \eqref{cubic}; the very last bound is obvious.
\end{proof}

\begin{rem}
By means of a different method, the spectral inclusion of Theorem~{\rm \ref{supernett}~(i)} was also shown in {\rm \cite[Theorem~4.2]{JT2}}, while Theorem~{\rm \ref{supernett}~(iii)} improves the corresponding statement of {\rm \cite[Theorem~4.2]{JT2}}.
\end{rem}

Note that 
due to \eqref{ineq}, $\mu>0$ implies $\beta>0$, and $\delta>0$ implies $\mu>0$ and thus $\beta>0$.
There\-fore if, in Theorem~\rmref{supernett}, (ii) applies then so does (i) and if (iii) applies, then so do (i) and~(ii).

\smallskip

In the following Proposition~\rmref{compare} we work out the precise form of the corresponding intersections
of the bounding sets in Theorem~\rmref{hendrik} and Theorem~\rmref{supernett} (i), (ii), and (iii).

Figures {\rm \ref{fig0}--\ref{fignew}} below illustrate how the spectral en\-clo\-sures by means of the quadratic numerical range 
(\redgrey) compare to those obtained by means of the numerical range (in light grey) 
and how the enclosures improve successively for the cases $\beta>0$, $\mu>0$, $\delta >0$, and $\beta\delta >4$.

\begin{figure}%
\def\onetwo{-1.0445}
\def\twothree{-3.1032}
\def\hnullLl{-2-0.871}
\def\hnullLr{-2-0.43646}
\def\hnullRl{-2+0.43646}
\def\hnullRr{-2+0.871}
\def\hnull{sqrt(-4*x/1.05/(4+x)-x^2)}
\def\hi{0.4*x/(2+x)}
\def\hii{0.4186*x}
\def\hiii{1.24+0.046*x+0.021*x^2}
\def\numr{-0.2*x+2*sqrt(-x/1.05)}
\newcommand\halfplane{%
\begin{scope}[very thin,color=lightgray,fill=lightgray]
  \addplot[domain=-4:-0.01,samples=2] {2} \closedcycle;
  \addplot[domain=-4:-0.01,samples=2] {-2} \closedcycle;
\end{scope}
  \addplot[domain=-4:-0.01,samples=2] {2};
  \addplot[domain=-4:-0.01,samples=2] {-2};
}
\newcommand\numrange{%
\begin{scope}[very thin,color=lightgray,fill=lightgray]
  \addplot[domain=-4:-0.01,samples=100] {-(\numr)} \closedcycle;
  \addplot[domain=-4:-0.01,samples=100] {\numr} \closedcycle;
\end{scope}
  \addplot[domain=-4:-0.01,samples=100,color=gray] {-(\numr)};
  \addplot[domain=-4:-0.01,samples=100,color=gray] {\numr};
  \draw[-] (axis cs:-4,-2) -- (axis cs:0,-2);
  \draw[-] (axis cs:-4,2) -- (axis cs:0,2);
  \draw[-] (axis cs:-4,-2) -- (axis cs:-4,2);
}

\thickmuskip=4mu plus 3mu minus 1mu
\setlength{\captionmargin}{3mm}
\begin{floatrow}
\ffigbox[\FBwidth]{\caption{\newline
Theorem~\ref{supernett} (i) with $k=0.2$, $\beta=4$, $\mu=0$, $\delta=0$.}
\label{fig0}}
{
\begin{tikzpicture}
\begin{axis}[xmin=-4,xmax=0.5,ymin=-2,ymax=2]
\pgfplotsset{ticks=none}
\halfplane
\begin{scope}[very thin,color=red,fill=red]
  \addplot[domain=-1.8:-0.01,samples=100] {-0.4*x/(2+x)} \closedcycle;
  \addplot[domain=-1.8:-0.01,samples=100] {0.4*x/(2+x)} \closedcycle;
  \addplot[domain=-4:-1.8,samples=2] {2} \closedcycle;
  \addplot[domain=-4:-1.8,samples=2] {-2} \closedcycle;
\end{scope}
  \addplot[domain=-1.8:-0.01,samples=100] {-0.4*x/(2+x)};
  \addplot[domain=-1.8:-0.01,samples=100] {0.4*x/(2+x)};
  \addplot[color=black,domain=-4:1,samples=2] {0};
\draw[-] (axis cs:0,-4) -- (axis cs:0,4);
\draw[thick,-] (axis cs:-2,-0.06) -- (axis cs:-2,0.06);
\draw[thick, -] (axis cs:-0.05,1) -- (axis cs:0.05,1);
\pgfplotsset{
    after end axis/.code={
	\node[below] at (axis cs:-2,-0.02){ $-2$ };
	\node[right] at (axis cs:0,1){ $1$ };
	}}
    \end{axis}
\end{tikzpicture}
}
\hspace{3mm}
\ffigbox[\FBwidth]{\caption{\small\newline
Theorem~\ref{supernett} (i), (ii) with $k=0.2$, $\beta=4$, $\mu=2.1$, $\delta=0$;
here $\lambda_{\rm i,ii} \approx 1.04$.}
\label{fig1}}
{
\begin{tikzpicture}
\begin{axis}[xmin=-4,xmax=0.5,ymin=-2,ymax=2]
\pgfplotsset{ticks=none}
\halfplane
\begin{scope}[very thin,color=red,fill=red]
  \addplot[domain=\onetwo:-0.01,samples=30] {-\hi} \closedcycle;
  \addplot[domain=\onetwo:-0.01,samples=30] {\hi} \closedcycle;
  \addplot[domain=-4:\onetwo,samples=2] {-\hii} \closedcycle;
  \addplot[domain=-4:\onetwo,samples=2] {\hii} \closedcycle;
\end{scope}
  \addplot[domain=\onetwo:-0.01,samples=30] {-\hi};
  \addplot[domain=-4:\onetwo,samples=2] {-\hii};
  \addplot[domain=-4:\onetwo,samples=2] {\hii};
  \addplot[domain=\onetwo:-0.01,samples=30] {\hi};
  \addplot[domain=-4:1,samples=2] {0};
\draw[thick, -] (axis cs:-2,-0.06) -- (axis cs:-2,0.06);
\draw[thick, -] (axis cs:-0.05,1) -- (axis cs:0.05,1);
\draw[-] (axis cs:0,-4) -- (axis cs:0,4);
\draw[thick, -] (axis cs:\onetwo,-0.06) -- (axis cs:\onetwo,0.06);
\pgfplotsset{
    after end axis/.code={
	\node[below] at (axis cs:-2,-0.02){ $-2$ };
	\node[right] at (axis cs:0,1){ $1$ };
	\node[below] at (axis cs:\onetwo,0.04){ \llap{$-$}$\lambda_{{\rm i,ii}}$ };
				}}
    \end{axis}
\end{tikzpicture}
}
\end{floatrow}
\vspace{1.2cm}
\begin{floatrow}
\ffigbox[\FBwidth]{\caption{\newline
Theorem~\ref{supernett}\,(i),\,(ii),\,(iii)\,without\,Theorem~\ref{hendrik} with $k=0.2$, $\beta=4$, $\mu=2.1$, $\delta=1.05$;
here $\lambda_{\rm i,ii} \approx 1.04$, $\lambda_{\rm ii,iii} \approx 3.10$, see Remark~\ref{newrem}.}
\label{fig2}}
{
\begin{tikzpicture}
\begin{axis}[xmin=-4,xmax=0.5,ymin=-2,ymax=2]
\pgfplotsset{ticks=none}
\numrange
\begin{scope}[very thin,color=red,fill=red]
  \addplot[domain=\onetwo:-0.01,samples=30] {-\hi} \closedcycle;
  \addplot[domain=\onetwo:-0.01,samples=30] {\hi} \closedcycle;
  \addplot[domain=\twothree:\onetwo,samples=2] {-\hii} \closedcycle;
  \addplot[domain=\twothree:\onetwo,samples=2] {\hii} \closedcycle;
  \addplot[domain=-4:\twothree,samples=10] {\hiii} \closedcycle;
  \addplot[domain=-4:\twothree,samples=10] {-(\hiii)} \closedcycle;
\end{scope}
  \addplot[domain=\onetwo:-0.01,samples=30] {-\hi};
  \addplot[domain=\twothree:\onetwo,samples=2] {-\hii};
  \addplot[domain=-4:\twothree,samples=10] {\hiii};
  \addplot[domain=-4:\twothree,samples=10] {-(\hiii)};
  \addplot[domain=\twothree:\onetwo,samples=2] {\hii};
  \addplot[domain=\onetwo:-0.01,samples=30] {\hi};
  \addplot[domain=-4:1,samples=2] {0};
\draw[thick, -] (axis cs:-2,-0.06) -- (axis cs:-2,0.06);
\draw[thick, -] (axis cs:-0.05,1) -- (axis cs:0.05,1);
\draw[-] (axis cs:0,-2) -- (axis cs:0,2);
\draw[thick, -] (axis cs:\onetwo,-0.06) -- (axis cs:\onetwo,0.06);
\draw[thick, -] (axis cs:\twothree,-0.06) -- (axis cs:\twothree,0.06);
\pgfplotsset{
   after end axis/.code={
	\node[below] at (axis cs:-2,-0.02){ $-2$ };
	\node[right] at (axis cs:0,1){ $1$ };
	\node[below] at (axis cs:\onetwo,0.04){ \llap{$-$}$\lambda_{{\rm i,ii}}$ };
	\node[below] at (axis cs:\twothree,0.04){ $-\lambda_{{\rm ii,iii}}\;$ };
				}}
    \end{axis}
\end{tikzpicture}
}
\hspace{3mm}
\ffigbox[\FBwidth]{\caption{\newline
Theorem~\ref{supernett} (i), (ii), (iii) and Theorem~\ref{hendrik} with $k=0.2$, $\beta=4$, $\mu=2.1$, $\delta=1.05$;
here $\beta\delta>4$, $k>\frac 4{\beta\delta}-1$, $\lambda_{\rm i,ii} \approx 1.04$, $\lambda_{\rm ii,iii} \approx 3.10$,
$I_{0,\mu}\approx(1.12,2.87)$, $I_0\approx(1.56,2.44)$, see Remark~\ref{newrem}.}
\label{fignew}}
{
\begin{tikzpicture}
\begin{axis}[xmin=-4,xmax=0.5,ymin=-2,ymax=2]
\pgfplotsset{ticks=none}
\numrange
\begin{scope}[very thin,color=red,fill=red]
  \addplot[domain=\onetwo:-0.01,samples=30] {-\hi} \closedcycle;
  \addplot[domain=\onetwo:-0.01,samples=30] {\hi} \closedcycle;
  \addplot[domain=\hnullRr:\onetwo,samples=2] {-\hii} \closedcycle;
  \addplot[domain=\hnullRr:\onetwo,samples=2] {\hii} \closedcycle;
  \addplot[domain=\hnullRl:\hnullRr,samples=100] {\hnull} \closedcycle;
  \addplot[domain=\hnullRl:\hnullRr,samples=100] {-\hnull} \closedcycle;
  \addplot[domain=\hnullLl:\hnullLr,samples=100] {\hnull} \closedcycle;
  \addplot[domain=\hnullLl:\hnullLr,samples=100] {-\hnull} \closedcycle;
  \addplot[domain=\twothree:\hnullLl,samples=2] {-\hii} \closedcycle;
  \addplot[domain=\twothree:\hnullLl,samples=2] {\hii} \closedcycle;
  \addplot[domain=-4:\twothree,samples=10] {\hiii} \closedcycle;
  \addplot[domain=-4:\twothree,samples=10] {-(\hiii)} \closedcycle;
\end{scope}
  \addplot[domain=\onetwo:-0.01,samples=30] {-\hi};
  \addplot[domain=\hnullRr:\onetwo,samples=2] {-\hii};
  \addplot[domain=\hnullRl:\hnullRr,samples=100] {\hnull};
  \addplot[domain=\hnullLl:\hnullLr,samples=100] {\hnull};
  \addplot[domain=\twothree:\hnullLl,samples=2] {-\hii};
  \addplot[domain=-4:\twothree,samples=10] {\hiii};
  \addplot[domain=-4:\twothree,samples=10] {-(\hiii)};
  \addplot[domain=\twothree:\hnullLl,samples=2] {\hii};
  \addplot[domain=\hnullLl:\hnullLr,samples=100] {-\hnull};
  \addplot[domain=\hnullRl:\hnullRr,samples=100] {-\hnull};
  \addplot[domain=\hnullRr:\onetwo,samples=2] {\hii};
  \addplot[domain=\onetwo:-0.01,samples=30] {\hi};
  \addplot[domain=-4:1,samples=2] {0};
\draw[thick, -] (axis cs:-2,-0.06) -- (axis cs:-2,0.06);
\draw[thick, -] (axis cs:-0.05,1) -- (axis cs:0.05,1);
\draw[-] (axis cs:0,-4) -- (axis cs:0,4);
\draw[thick, -] (axis cs:\onetwo,-0.06) -- (axis cs:\onetwo,0.06);
\draw[thick, -] (axis cs:\twothree,-0.06) -- (axis cs:\twothree,0.06);
\pgfplotsset{
   after end axis/.code={
	\node[below] at (axis cs:-2,-0.02){ $-2$ };
	\node[right] at (axis cs:0,1){ $1$ };
	\node[below] at (axis cs:\onetwo,0.04){ \llap{$-$}$\lambda_{{\rm i,ii}}$ };
	\node[below] at (axis cs:\twothree,0.04){ $-\lambda_{{\rm ii,iii}}\;$ };
				}}
    \end{axis}
\end{tikzpicture}
}
\end{floatrow}
\vspace{1cm}
\begin{center}
Figures \ref{fig0}--\ref{fignew}: Spectral enclosures obtained from $W(\cA)$ (light grey) and \\
from $W^2(\cA)$ (\redgrey).
\end{center}
\end{figure}

\begin{prop}
\label{compare}
Suppose that condition \eqref{Simon1} holds and define
\begin{alignat*}{2}
\lambda_{\rm i,ii} &:= \frac \beta 2 \Bigl( 1 - \frac k{k_\mu} \Bigr) \in \Bigl[0,\frac\beta 2\Bigr], \qquad
& & \qquad \text{if $\mu\!>\!0$ $($which implies $\beta\!>\!0)$},\\
\lambda_{\rm ii,iii} &:=
\begin{cases}
\displaystyle
\frac{\mu^2}{2\delta}\Bigl( 1 + \frac k{k_\mu}\Bigr) \in \Bigl[\frac \beta 2, \frac{\mu^2}{\delta} \Bigr), & k_\mu > k,\\
\hspace{9mm} \infty, & k_\mu=k=0,
\end{cases}
& & \qquad \text{if $\delta\!>\!0$ $($which implies $\mu\!>\!0$ and $\beta\!>\!0)$}.
\end{alignat*}
Then the spectrum of $\cA$ satisfies the following inclusions:
\begin{enumerate}
\item[{\rm (a)}]
if $\mu > 0$ $($and hence $\beta>0)$, then
\[
  \sigma(\cA)
  \subset
\delimitershortfall-1pt
\left\{\lambda \in \C \mid -\gamma \le \re \lambda < 0, \
|\im \lambda| \le\!
\begin{cases}
\leavevmode\\[-4ex]
\dfrac 1{ 1\!-\!\frac 2\beta |\re\lambda|} \,k|\re\lambda|, & 0 \!<\!|\re\lambda| \!\le\!\lambda_{\rm i,ii} \\
\hspace{6mm} k_\mu |\re\lambda|, & \lambda_{\rm i,ii} \!<\!|\re\lambda| \!\le\!\gamma \\[-4pt]
\end{cases}
\right\}\!;
\]
\item[{\rm (b)}]
if $\delta > 0$ $($and hence $\mu>0$, $\beta>0)$,
then $\lambda_{\rm i,ii} 
\le\lambda_{\rm ii,iii}$ and
\[
\sigma(\cA)
\!\subset\mkern-5mu
\left\{ \!\lambda \!\in\! \C \mid -\gamma \!\le\! \re \lambda \!<\! 0, \: |\re\lambda| \!\notin\! I_0, \:
|\im \lambda| \!\le\mkern-5mu
\begin{cases}
\!\dfrac 1{ 1\!-\! \frac 2\beta |\re\lambda|} \,k|\re\lambda|, & |\re\lambda| \!\in\! [0,\lambda_{\rm i,ii}]\hspace{-2mm} \\[-1mm]
\hspace{4mm} k_\mu |\re\lambda|,  &
\hspace{-10mm}|\re\lambda|\!\in\!(\lambda_{\rm i,ii},\lambda_{\rm ii,iii})\!\setminus\! I_{0,\mu}\!\! \\[0.5mm]
\hspace{4mm} h_0(|\re\lambda|),  &
|\re\lambda|\!\in\! I_{0,\mu}\!\setminus\!I_0,\\[0.5mm]
\hspace{4mm} h_{\rm iii}(|\re\lambda|),
&  |\re\lambda| \!\in\![\lambda_{\rm ii,iii},\gamma)\hspace{-1mm}
\end{cases}
\right\}\!
\]
where $I_0$, $h_0$ are as defined in Theorem {\rm \ref{hendrik}}, $k_\mu$, $h_{\rm iii}$ are as defined in Theorem {\rm\ref{supernett}}, and
$I_{0,\mu}$ is a $($possibly empty$)$ interval centred at $\frac \beta 2$, $I_0 \subset I_{0,\mu} \subset \bigl( \lambda_{\rm i,ii}, \lambda_{\rm ii,iii} \bigr)$, given by
\[
 I_{0,\mu}  := \begin{cases}
                \hspace{2.7cm}\emptyset & \mbox{ if } \ k_\mu^2  \le \frac 4{\beta\delta}-1,\\
								\left( \frac \beta 2 \Bigl( 1 - \sqrt{1 \!-\! \frac 4{\beta\delta} \frac 1{k_\mu^2+1}}\, \Bigr),
								       \frac \beta 2 \Bigl( 1 + \sqrt{1 \!-\! \frac 4{\beta\delta} \frac 1{k_\mu^2+1}}\, \Bigr) \right)  &
								\mbox{ if } \  k_\mu^2  > \frac 4{\beta\delta}-1,
                \end{cases}
\]
which satisfies $I_0 = I_{0,\mu}$ if and only if $k_\mu=0$ and $I_{0,\mu} = \bigl( \lambda_{\rm i,ii}, \lambda_{\rm ii,iii} \bigr)$ if and only if $\mu^2=\beta\delta$.
\end{enumerate}
\end{prop}

\begin{rem}
\label{newrem}
If
the interval 
$I_{0,\mu}$
is non-empty, then Theorem~\ref{hendrik} gives an improvement of Theorem \ref{supernett}.
This improvement is most substantial if even $I_0$ is non-empty.

In fact, $I_0 \ne \emptyset$ if and only if $\beta\delta > 4$, see \eqref{I0};
in this case Theorem \ref{hendrik} yields a spectral free strip for $|\re\lambda|\in I_0$ which is not provided by Theorem \ref{supernett}~(ii).
Further, $I_{0,\mu} \ne \emptyset$ if and only if $k_\mu^2  > \frac 4{\beta\delta}-1$;
in this case Theorem \ref{hendrik} yields a better estimate than Theorem \ref{supernett}~(ii) for $|\re\lambda| \in I_{0,\mu} \subset \bigl( \lambda_{\rm i,ii}, \lambda_{\rm ii,iii} \bigr)$.

Independently of $\mu$, there is always an improvement if $\beta\delta>4$ since then $I_0\ne \emptyset$.
Similarly,
if $\beta\delta < 4$ and $k^2  \ge \frac 4{\beta\delta}-1$,
then $I_{0,\mu} \ne \emptyset$ since $k_\mu >k$ due to Remark~\ref{kmu-rem}~(a); 
the same applies if $\beta\delta = 4$ and $k>0$. This is illustrated in Figure \ref{fignew}.

Depending on $\mu$, for $\beta$, $\delta$ fixed, the interval $I_{0,\mu}$ is decreasing for increasing $\mu$ since $k_\mu$ decreases, see 
Remark~\ref{kmu-rem}~(a).
More precisely, since $\mu^2 \ge \beta\delta$ by \eqref{ineq}, starting from $I_{0,\sqrt{\beta\delta}} = \bigl( \lambda_{\rm i,ii}, \lambda_{\rm ii,iii} \bigr)$ for $\mu^2=\beta\delta$, the interval $I_{0,\mu}$ shrinks down to a (possibly empty) limiting interval $I_{0,\infty}$ obtained from $I_{0,\mu}$ by replacing $k_\mu$ by its limit $\lim_{\mu\to\infty} k_\mu =k$. For $k>0$, we have
\[
   \begin{cases}
	 I_{0,\mu} \varsupsetneq I_{0, \infty} \varsupsetneq  I_0 & \ \mbox{ if } \ k^2  \ge \frac 4{\beta\delta}-1, \\
   I_{0,\mu} \varsupsetneq I_0 = \emptyset, \ \mu \in (0,\mu_0),
	 \quad\ I_{0,\mu} = I_{0, \infty} = I_0 = \emptyset, \ \mu \in [\mu_0,\infty), & \ \mbox{ if } \ k^2  < \frac 4{\beta\delta}-1,
	 \end{cases}
\]
where $\mu_0\in (0,\infty)$ is the threshold where $k_{\mu_0}^2  = \smash{\frac 4{\beta\delta}}-1$.
For $k=0$, we always have $I_{0,\mu}=I_0$,
and this interval is non-empty if and only if $\beta\delta>4$;
see also Theorem \ref{nett}~(iii) and Figures \ref{fig5}, \ref{fig6}.
\end{rem}

\begin{proof}[Proof of Proposition~\rmref{compare}.]

A straightforward computation shows that, for $t\in(0,\frac\beta 2)$, we have
$h_{\rm i}(t) \le h_{\rm ii}(t)$ if and only if $t \le \frac\beta2 \bigl(1-\frac{k}{k_\mu}\bigr) = \lambda_{\rm i,ii}$.

To compare the functions $h_{\rm ii}$ and $h_{\rm iii}$ we consider
the equation \eqref{cubic} defining $h_{\rm iii}(t)$ with $y$ replaced by $h_{\rm ii}(t) = k_\mu t$, which leads to the equation
\begin{equation}\label{ii-iii}
  (k_\mu^2+1) (k_\mu-k) t = \frac 2 \delta k_\mu.
\end{equation}
By definition \eqref{kmu}, $k_\mu$ satisfies
\begin{equation}\label{kmu-eq}
  0 = k_\mu^2 \Bigl( k_\mu^2 + 1 - k^2 - \frac 4{\mu^2} \Bigr) - k^2
    = k_\mu^4 + k_\mu^2 - k_\mu^2 k^2 - \frac 4{\mu^2} k_\mu^2 - k^2
    = (k_\mu^2+1) (k_\mu^2-k^2) - \frac 4{\mu^2} k_\mu^2.
\end{equation}
Therefore, if $k_\mu>k$, we obtain a unique solution of~\eqref{ii-iii},
\[
  \lambda_{\rm ii,iii} = \frac 2 \delta \frac{k_\mu}{(k_\mu^2+1) (k_\mu-k)} = \frac 2 \delta \frac{k_\mu(k_\mu+k)}{\frac 4{\mu^2}  k_\mu^2}
  = \frac {\mu^2}{2\delta} \frac{k_\mu+k}{k_\mu} = \frac {\mu^2}{2\delta} \Bigl( 1+ \frac{k}{k_\mu} \Bigr) \le  \frac{\mu^2}\delta\,,
\]
for which $h_{\rm ii}(t) \le h_{\rm iii}(t)$ if and only if $t \le \lambda_{\rm ii,iii}$.
If $k_\mu=k$, then $k_\mu=k=0$ due to Remark~\ref{kmu-rem}~(a) and thus, in this case,
$h_{\rm ii}(t)=0$ for all $t\in[0,\infty)$ and $\lambda_{\rm ii,iii} = \infty$.

Since $\beta \delta \le \mu^2$ by \eqref{ineq}, it is easy to see that $\lambda_{\rm i,ii} \le \frac \beta 2 \le \frac{\mu^2}{2\delta} \le \lambda_{\rm ii,iii}$ and hence
\begin{alignat*}{2}
 &h_{\rm i}(t) \le h_{\rm ii}(t) \le h_{\rm iii}(t), \quad &&  t\in [0, \lambda_{\rm i,ii}], \\
 &h_{\rm ii}(t) \le \min\{ h_{\rm i}(t),h_{\rm iii}(t)\}, \quad &&  t\in [\lambda_{\rm i,ii},\lambda_{\rm ii,iii}], \\
 &h_{\rm iii}(t) \le h_{\rm ii}(t) \le h_{\rm i}(t), \quad && t\in [\lambda_{\rm ii,iii}, \infty).
\end{alignat*}

Finally, if $\delta>0$, we compare the enclosures of Theorem \ref{supernett} with Theorem \ref{hendrik}.
It is not difficult to see that, for $t\in [0,\beta)$,
\[
  h_0(t) \le h_{\rm ii}(t) = k_\mu t \ \iff \ t \in I_{0,\mu}.
\]
Since $\smash{\frac 1{k_\mu^2+1}} \le 1$, it is obvious that $I_0 \subset I_{0,\mu}$ and $I_{0,\mu} = I_0$ if and only if $k_\mu=0$.
By~\eqref{kmu-eq} one obtains $\frac{k}{k_\mu} = \sqrt{1-\frac{4}{\mu^2}\frac{1}{k_\mu^2+1}}$;
since $\mu^2\ge\beta\delta$, it follows that $I_{0,\mu} = \bigl( \lambda_{\rm i,ii}, \lambda_{\rm ii,iii} \bigr)$ if and only if $\mu^2=\beta\delta$. 
Now the inclusion $I_{0,\mu} \subset \bigl( \lambda_{\rm i,ii}, \lambda_{\rm ii,iii} \bigr)$
for $\mu^2>\beta\delta$ follows if we recall that $I_{0,\mu}$ is decreasing for increasing $\mu$, see Remark~\ref{newrem}.
\end{proof}

\section{Self-adjoint damping: estimates for QNR and spectrum} 
\label{section7}

In this section we assume that the damping operator is not only sectorial but even self-adjoint, i.e.\  $A_0^{-\frac 12}DA_0^{-\frac 12}$ is self-adjoint. In this case, it is known, see \cite[Proof of Lemma~4.5]{TWII}, that the operator $\cA$ is $\cJ$-self-adjoint, i.e.\  $\cA^*=\cJ \cA \cJ$ with
\[
  \cJ=\begin{mat}
       I_{\hhalf} & 0 \\ 0 & -I_{H}
      \end{mat} \quad \mbox{in } H_{\frac 12} \times H.
\]
Hence the spectrum $\sigma(\cA)$ of $\cA$ is symmetric with respect to the real line, see \cite[Satz I.2]{L62}.
This property is reflected by both the numerical range and the quadratic numerical range. 

\vspace{-1mm}

\begin{lem}\label{Wolfach}
Assume $A_0^{-\frac 12}D A_0^{-\frac 12}$ is a bounded self-adjoint operator in $H$.
Then
$W(\cA|_{H_1\times H_1})$ and $W^2(\cA|_{H_1\times H_1})$ are symmetric with respect to the real line,
and hence so are $\overline{W(\cA)} = \overline{W(\cA|_{H_1\times H_1})}$ and $\overline{W^2(\cA|_{H_1\times H_1})}$.
Moreover, $\sigma(\cA) \subset \overline{W^2(\cA|_{H_1\times H_1})}$.
\end{lem}

\begin{proof}
We have $\lambda \in W(\cA|_{H_1\times H_1})$ if and only if there is $(f, g)^{\top} \!\in  H_1\times H_1 $ with \vspace{-1.5mm}
$\|f\|^2_{\hhalf}\! + \|g\|^2 = 1$ so that \eqref{WA} holds. Clearly, $(f, -g)^{\top} \in  H_1\times H_1$
and \eqref{WA} shows that $\overline{\lambda}$ is in $W(\cA|_{H_1\times H_1})$.

The symmetry of $W^2(\cA|_{H_1\times H_1})$ follows from the fact that, for self-adjoint $D$ and $(f,g)^{\top} \! \in  H_1\times H_1$ with $\|f\|_{\hhalf}  = \|g\|= 1$,
the polynomial $\det(\cA_{f,g}-\lambda) = \lambda^2 + \lambda\la Dg,g\ra + \smash{|\la f,g\ra_{\hhalf}|^2}$
has real coefficients and so its zeros are symmetric with respect to the real~line.

For the next claim it remains to be noted that $H_1\times H_1$ is a core (cf.\ Proposition~\ref{Ainv}) and thus
$\overline{W(\cA)} = \overline{W(\cA|_{H_1\times H_1})}$ by \cite[Problem \V.3.7]{K}.

Finally, let $\lambda \in \sigma(\cA)$.
Then either $\lambda \in \sigma_{ap}(\cA)$
and hence $\lambda \in \overline{W^2(\cA|_{H_1\times H_1})}$ by Theorem~\ref{Karl},
or $\lambda \in \sigma_{r}(\cA)$.
In the latter case we obtain
$\overline{\lambda} \in \sigma_{p}(\cA) \subset \overline{W^2(\cA|_{H_1\times H_1})}$
by \cite[Theorem~VI.6.1]{bog} since $\cA$ is $\cJ$-selfadjoint.
Hence $\lambda \in \overline{W^2(\cA|_{H_1\times H_1})}$ by the symmetry shown before.
\end{proof}

\begin{thm}\label{nett}
Assume $A_0^{-\frac 12}D A_0^{-\frac 12}$ is a bounded self-adjoint operator in $H$.
\begin{enumerate}
\item[{\rm (i)}] If \,$\beta >0$, then \vsqueeze{1mm}
\begin{equation} \label{Pader0i}
\sigma(\cA) \subset \Bigl\{\lambda \in \C \mid \re\lambda \leq
-\frac{\beta}{2} \Bigr\} \cup \Bigl[-\frac{\beta}{2},0 \Bigr).
\end{equation}
\item[{\rm (ii)}] If \,$0<\mu < 2$, then $\beta \ge \mu a_0 > 0$ and
\begin{equation}
\label{Paderi}
\sigma(\cA) \subset \Bigl\{\lambda \in \C \mid \re \lambda \leq -\frac \beta 2,\
|\im\lambda| \leq \smash{\frac{\sqrt{4-\mu^2}}{\mu}}\, |\re \lambda | \Bigr\} \cup \Bigl[-\frac{\beta}{2},0 \Bigr);
\end{equation}
if \,$\mu \ge 2$, then
\[
\sigma(\cA) 
\subset (-\infty,0).
\]
\item[{\rm (iii)}] \,If $\,\delta>0$, then $\beta \ge \delta a_0^2 > 0$,
$\sigma(\cA) \setminus \R$ is bounded and confined to a part of a disk, and
\begin{equation}
\label{Paderiii}
\sigma(\cA) \subset 
\Bigl(-\infty,-\frac 2\delta\,\Bigr] \cup
\Bigl\{\lambda \in \C \mid
-\frac 2\delta \le \re \lambda \leq -\frac{\beta}{2},\ \Bigl|\lambda + \frac{1}{\delta}\Bigr| \leq \frac{1}{\delta}\Bigr\}\cup
\Bigl[-\frac{\beta}{2}, 0\Bigr);
\end{equation}
if $\,\beta \delta > 4$, then \vsqueeze{1.5mm}
\[
  \sigma(\cA) \subset
  \Bigl( -\infty, -\frac{\beta}{2} \Bigl( 1 + \sqrt{ 1 - \frac 4{\beta\delta}}\,\Bigr) \, \Bigr] \cup
  \Bigl[ -\frac{\beta}{2} \Bigl( 1 - \sqrt{ 1 - \frac 4{\beta\delta}}\,\Bigr) \,, 0 \Bigr).
\]
\end{enumerate}
If, in any of the above cases, in addition $\gamma<\infty$, then also
\begin{equation}
\label{Paderiv}
\sigma(\cA) \subset \Bigl[-\gamma,-\frac{\gamma}{2}\,\Bigr] \cup \Bigl\{\lambda \in \C \mid
-\frac{\gamma}{2} \leq \re \lambda < 0 \Bigr\}.
 \end{equation}
\end{thm}

\begin{proof}
The self-adjointness of $A_0^{-\frac 12}D A_0^{-\frac 12}$ implies that $\im \la Dg,g\ra=0$ for $g\in H_1$
and hence \eqref{Simon1} holds with $k=0$.

\begin{figure}%
\thickmuskip=4mu plus 3mu minus 1mu
\setlength{\captionmargin}{3mm}%
\def\twothree{-1.40625}%
\def\kmu{1.76383}%
\newcommand\halfplane{%
\begin{scope}[very thin,color=lightgray,fill=lightgray]
  \addplot[domain=-3:-0.01,samples=2] {3} \closedcycle;
  \addplot[domain=-3:-0.01,samples=2] {-3} \closedcycle;
\end{scope}
  \addplot[domain=-3:-0.01,samples=2] {3};
  \addplot[domain=-3:-0.01,samples=2] {-3};
}
\newcommand\numrange[1]{%
\def\numr{2*sqrt(-x/#1)}
\begin{scope}[very thin,color=lightgray,fill=lightgray]
  \addplot[domain=-3:-0.01,samples=100] {-(\numr)} \closedcycle;
  \addplot[domain=-3:-0.01,samples=100] {\numr} \closedcycle;
\end{scope}
  \addplot[domain=-3:-0.01,samples=100,color=gray] {-(\numr)};
  \addplot[domain=-3:-0.01,samples=100,color=gray] {\numr};
  \draw[-] (axis cs:-3,-3) -- (axis cs:0,-3);
  \draw[-] (axis cs:-3,3) -- (axis cs:0,3);
  \draw[-] (axis cs:-3,-3) -- (axis cs:-4,3);
}

\begin{floatrow}
\ffigbox[\FBwidth]
{\caption{\label{fig3} \small \newline Theorem~\ref{nett} (i) with $\beta=4$, $\mu=0$, $\delta=0$.}}
{
\begin{tikzpicture}
\begin{axis}[xmin=-3,xmax=0.5,ymin=-3,ymax=3]
\pgfplotsset{ticks=none}
\halfplane
\addplot[ultra thick,color=red,domain=-3:0,samples=2] {0};
\begin{scope}[very thin,color=red,fill=red]
\addplot[domain=-3:-1,samples=2] {3} \closedcycle;
\addplot[domain=-3:-1,samples=2] {-3} \closedcycle;
\end{scope}
\addplot[domain=-3:1,samples=2] {0};
\draw[-] (axis cs:0,-4) -- (axis cs:0,4);
\draw[thick, -] (axis cs:-1,-0.1) -- (axis cs:-1,0.1);
\draw[-] (axis cs:-1,-3) -- (axis cs:-1,3);
\pgfplotsset{
    after end axis/.code={
	\node[above] at (axis cs:-1,-0.6){ $-2$ };
	}}
\end{axis}
\end{tikzpicture}
}
\hspace{3mm}
\ffigbox[\FBwidth]
{\caption{\label{fig4} \small \newline Theorem~\ref{nett} (i), (ii) with $\beta=4$, $\mu=1.5$, $\delta=0$.}}
{
\begin{tikzpicture}
\begin{axis}[xmin=-3,xmax=0.5,ymin=-3,ymax=3]
\pgfplotsset{ticks=none}
\halfplane
\addplot[ultra thick,color=red,domain=-3:0,samples=2] {0};
\begin{scope}[very thin,color=red,fill=red]
\addplot[domain=-3:-1,samples=2] {\kmu*x} \closedcycle;
\addplot[domain=-3:-1,samples=2] {-\kmu*x} \closedcycle;
\end{scope}
\addplot[domain=-3:1,samples=2] {0};
\draw[-] (axis cs:0,-4) -- (axis cs:0,4);
\draw[thick, -] (axis cs:-1,-0.1) -- (axis cs:-1,0.1);
\draw[thick, -] (axis cs:-0.05,1) -- (axis cs:0.05,1);
\draw[-] (axis cs:-1,-1.76) -- (axis cs:-1,1.76);
\addplot[color=black,domain=-2:-1,samples=2] {\kmu*x};
\addplot[domain=-2:-1,samples=2] {-\kmu*x};
\pgfplotsset{
    after end axis/.code={
	\node[above] at (axis cs:-1,-0.6){ $-2$ };
	\node[above] at (axis cs:0.13,0.75){ $1$ };
	}}
\end{axis}
\end{tikzpicture}
}
\end{floatrow}
\vspace{1.2cm}
\begin{floatrow}
\ffigbox[\FBwidth]
{\caption{\label{fig5} \small \newline Theorem~\ref{nett} (i), (ii), (iii) with $\beta=4$, $\mu=1.5$, $\delta=0.4$;
here $k_\mu^2 = \frac79 < \frac32 = \frac{4}{\beta\delta}-1$, $I_0=I_{0,\mu}=\emptyset$,
so no improvement by Theorem~\ref{hendrik}, see Remark~\ref{newrem}.
}}
{
\begin{tikzpicture}
\begin{axis}[xmin=-3,xmax=0.5,ymin=-3,ymax=3]
\pgfplotsset{ticks=none}
\numrange{0.4}
\addplot[ultra thick,color=red,domain=-3:0,samples=2] {0};
\begin{scope}[very thin,color=red,fill=red]
\addplot[domain=\twothree:-1,samples=2] {\kmu*x} \closedcycle;
\addplot[domain=\twothree:-1,samples=2] {-\kmu*x} \closedcycle;
\addplot[domain=-2.5:\twothree,samples=100] {2*sqrt(-x)*sqrt(2.5+x)}\closedcycle;
\addplot[domain=-2.5:\twothree,samples=100] {-2*sqrt(-x)*sqrt(2.5+x)}\closedcycle;
\end{scope}
\addplot[domain=-2.5:\twothree,samples=150] {2*sqrt(-x)*sqrt(2.5+x)};
\addplot[domain=-2.5:\twothree,samples=150] {-2*sqrt(-x)*sqrt(2.5+x)};
\addplot[domain=\twothree:-1,samples=2] {\kmu*x};
\addplot[domain=\twothree:-1,samples=2] {-\kmu*x};
\draw[-] (axis cs:-1,-1.76) -- (axis cs:-1,1.76);
\addplot[domain=-3:1,samples=2] {0};
\draw[-] (axis cs:0,-4) -- (axis cs:0,4);
\draw[thick, -] (axis cs:-0.05,1) -- (axis cs:0.05,1);
\draw[thick, -] (axis cs:-1,-0.1) -- (axis cs:-1,0.1);
pgfplotsset{
   after end axis/.code={
	\node[above] at (axis cs:-1,-0.6){ $-2$ };
	\node[above] at (axis cs:0.13,0.75){ $1$ };
	}}
\end{axis}
\end{tikzpicture}
}
\hspace{3mm}
\ffigbox[\FBwidth]
{\caption{\label{fig6} \small \newline Theorem~\ref{nett} (iii) with $\beta=4$, $\delta=\frac43$;
here spectral gap in $-I_0=-I_{0,\mu}=(-3,-1)$ by Theorem \ref{hendrik}, see Remark~\ref{newrem}.
}}
{
\begin{tikzpicture}
\begin{axis}[xmin=-3,xmax=0.5,ymin=-3,ymax=3]
\pgfplotsset{ticks=none}
\numrange{4/3}
\addplot[ultra thick,color=red,domain=-3:-1.5,samples=2] {0};
\addplot[ultra thick,color=red,domain=-0.5:0,samples=2] {0};
\addplot[domain=-3:1,samples=2] {0};
\draw[-] (axis cs:0,-4) -- (axis cs:0,4);
\draw[thick, -] (axis cs:-0.05,1) -- (axis cs:0.05,1);
\draw[thick, -] (axis cs:-1,-0.1) -- (axis cs:-1,0.1);
pgfplotsset{
   after end axis/.code={
	\node[above] at (axis cs:-1,-0.6){ $-2$ };
	\node[above] at (axis cs:0.13,0.75){ $1$ };
	}}
\end{axis}
\end{tikzpicture}
}
\end{floatrow}
\vspace{1cm}
\begin{center}
Figures \ref{fig3}--\ref{fig6}: Spectral enclosures obtained from $W(\cA)$ (light grey) 
and \\ from $W^2(\cA)$ (\redgrey).
\end{center}
\end{figure}

(i) The inclusion \eqref{Pader0i} follows from Theorem~\ref{supernett}~(i) since $k=0$ implies $h_{\rm i}(t)=0$ for all $t\in[0,\frac \beta 2)$.

(ii)
If we use $\beta \ge \mu a_0$, see \eqref{ineq}, and formula \eqref{kmuk0} which describes $k_\mu$ in the case $k=0$,
both inclusions follow from part~(i) and Theorem~\ref{supernett}~(ii).

(iii) By \eqref{ineq} we have $\beta \ge \delta a_0^2$. Further,
for $k=0$, the equation $\eqref{cubic}$ defining $h_{\rm iii}$
reads $(y^2+t^2)y = \frac2\delta ty$.
Thus, $h_{\rm iii}(t) = \sqrt{\frac 2\delta t - t^2}$ for $t \in [0,\frac2\delta]$
and $h_{\rm iii}(t) = 0$ for $t>\frac2\delta$.
Now both assertions in (iii) follow from part~(i), Theorem~\ref{supernett}~(iii) and Theorem~\ref{hendrik}.

By Lemma~\ref{Wolfach} it suffices to prove the inclusion in \eqref{Paderiv} for
$\overline{W^2(\cA|_{H_1\times H_1})} \setminus\{0\}$ in place of $\sigma(\cA)$.
Let $\lambda\in W^2(\cA|_{H_1\times H_1}) \setminus \R$.
Then there exists $(f,g)^{\top} \!\in  H_1\times H_1$ with \vspace{-0.5mm} $\|f\|_{\hhalf}\mkern-1.5mu = \|g\|= 1$
such that \eqref{Simon3} and hence \eqref{hendrik1}, \eqref{hendrik2} hold.
Using $\im\lambda\ne 0$ and $\im \la Dg,g\ra=0$ in \eqref{chr1} we find
\[
  |\re \lambda| = \frac{1}{2} \la Dg,g\ra \le \frac{\gamma}{2}\,.
\]
Then, by Proposition~\ref{qnrincl}, we conclude that
\[
W^2(\cA|_{H_1\times H_1}) \subset \Bigl[-\gamma,-\frac{\gamma}{2}\,\Bigr] \cup \Bigl\{\lambda \in \C \mid
-\frac{\gamma}{2} \leq \re \lambda < 0 \Bigr\}. \qedhere
\]
\end{proof}

\begin{rem}
We mention that, by means of a different method, the inclusions in
{\rm (i)}, the second inclusion in {\rm (ii)}, and the first inclusion in {\rm (iii)} were shown in {\rm \cite[Theorem~3.3]{JT}}, while \eqref{Paderiv} is an improvement of a corresponding inclusion therein.
\end{rem}

As in the previous section, due to \eqref{ineq}, $\mu>0$ implies $\beta>0$, and $\delta>0$ implies $\mu>0$ and thus $\beta>0$.
Therefore if, in Theorem~\rmref{nett}, (ii) applies then so does (i) and if (iii) applies, then so do (i) and (ii).
The precise form of the combination of all inclusions is given in the next proposition.


Figures~\ref{fig3}--\ref{fig6} illustrate how the spectral enclosures by means of the quadratic numerical range (\redgrey) 
compare to those obtained by means of the numerical range (light grey) and how the enclosures
improve successively for the cases $\beta>0$, $0<\mu<2$, $\delta>0$ and $\beta\delta>4$.


\begin{prop}
\label{compare-sa}
Let $A_0^{-\frac 12}D A_0^{-\frac 12}$ be a bounded self-adjoint operator in~$H$.
\begin{enumerate}
\item[{\rm (a)}]
If \,$\mu\ge 2$, then 
\[
 \sigma(\cA) \subset \begin{cases} (-\infty, 0) & \mbox{if } \gamma\!=\!\infty,\\
                                   \hspace{2.1mm} [-\gamma, 0) & \mbox{if } \gamma\!<\!\infty;
                     \end{cases}
\]
if, in addition, $\delta>0$ $($and hence $\beta>0)$ with $\beta\delta > 4$, then
\[
 \sigma(\cA) \subset
 \begin{cases}
 \Bigl(\!-\infty, - \dfrac \beta 2 \Bigl( 1 + \sqrt{1\!-\!\dfrac 4{\beta\delta}} \,\Bigr)\Bigr]
 \cup \Bigl[- \dfrac \beta 2 \Bigl( 1 - \sqrt{1\!-\!\dfrac 4{\beta\delta}} \,\Bigr),0\Bigr) & \mbox{if } \gamma\!=\!\infty,\\[1.6ex]
 \hspace{1.5mm} \Bigl[-\gamma, - \dfrac \beta 2 \Bigl( 1 + \sqrt{1\!-\!\dfrac 4{\beta\delta}} \,\Bigr)\Bigr]
 \cup \Bigl[- \dfrac \beta 2 \Bigl( 1 - \sqrt{1\!-\!\dfrac 4{\beta\delta}} \,\Bigr),0\Bigr) & \mbox{if } \gamma\!<\!\infty.
 \end{cases}
\pagebreak[2]
\]
\item[{\rm (b)}]
\newcommand\scup{\mkern-2mu\cup\mkern-2mu}
If \,$0<\mu<2$ $($and hence $\beta>0)$, then
\[\hskip-1.2em
\sigma(\cA) \!\subset\!
\begin{cases}
\hspace{1.922cm}\Bigl\{\lambda \!\in\! \C \mid -\infty \!\le\! \re \lambda \!\leq\! -\dfrac{\beta}{2}, \;
|\im\lambda| \!\leq\! \dfrac{\sqrt{4\!-\!\mu^2}}{\mu} |\re \lambda | \Bigr\} \scup \Bigl[-\dfrac{\beta}{2},0 \Bigr) & \mbox{if } \gamma\!=\!\infty, \\[1.4ex]
 \left[-\gamma,-\dfrac{\gamma}{2}\right] \scup
\Bigl\{\lambda \!\in\! \C \mid -\dfrac \gamma 2 \!\le\!  \re \lambda \!\leq\! -\dfrac{\beta}{2}, \;
|\im\lambda| \!\leq\! \dfrac{\sqrt{4\!-\!\mu^2}}{\mu} |\re \lambda | \Bigr\} \scup \Bigl[-\dfrac{\beta}{2},0 \Bigr) & \mbox{if } \gamma\!<\!\infty;
\end{cases}
\]
if, in addition to $0<\mu<2$, also $\delta>0$, then
\[\hskip-1.2em
 \sigma(\cA) \!\subset\!
 \begin{cases}
    \!
    \Bigl(\!-\infty,-\dfrac 2\delta\Bigr) \scup \Bigl\{\lambda \!\in\! \C \!\mid\! - \dfrac 2\delta \!\le\! \re \lambda \!\leq\! -\dfrac {\mu^2}{2\delta}, \,
    \Bigl|\lambda \!+\! \dfrac{1}{\delta}\Bigr| \!\leq\! \dfrac{1}{\delta}\Bigr\} & \\[1.3ex]
    \,{}\cup \Bigl\{\lambda \!\in\! \C \!\mid\!  - \dfrac {\mu^2}{2\delta} \!\le\! \re \lambda \!\leq\! - \dfrac{\beta}{2}, \,
    |\im\lambda| \!\leq\! \dfrac{\sqrt{4\!-\!\mu^2}}{\mu} |\re \lambda | \Bigr\} \scup \Bigl[-\dfrac{\beta}{2},0 \Bigr)
    & \mbox{if } \gamma\!=\!\infty,\\[3.5mm]
    \!
    \Bigl[-\gamma,-\!\min\Bigl\{\dfrac 2\delta,\dfrac \gamma 2\Bigr\}\Bigr) \scup \Bigl\{\lambda \!\in\! \C \!\mid\! -\!\min\Bigl\{\dfrac 2\delta,\dfrac \gamma 2 \Bigr\}\!\le\!  \re \lambda \!\leq\! -\dfrac {\mu^2}{2\delta}, \,
    \Bigl|\lambda \!+\! \dfrac{1}{\delta}\Bigr| \!\leq\! \dfrac{1}{\delta}\Bigr\} \hspace{-9mm}& \\[1.4ex]
    \,{}\cup \Bigl\{\lambda \!\in\! \C \!\mid\!  - \dfrac {\mu^2}{2\delta} \!\le\! \re \lambda \!\leq\! - \dfrac{\beta}{2}, \,
    |\im\lambda| \!\leq\! \dfrac{\sqrt{4\!-\!\mu^2}}{\mu} |\re \lambda | \Bigr\} \scup \Bigl[-\dfrac{\beta}{2},0 \Bigr)
    & \mbox{if }\gamma\!<\!\infty,\, \mu^2 \mkern-4mu<\mkern-4mu \gamma\delta,
		\\[3.5mm]
    \!
    \Bigl[-\gamma,-\dfrac \gamma 2\Bigr) & \\
    \,{}\cup \Bigl\{\lambda \!\in\! \C \!\mid\! - \dfrac \gamma 2 \!\le\! \re \lambda \!\leq\! - \dfrac{\beta}{2}, \,
    |\im\lambda| \!\leq\! \dfrac{\sqrt{4\!-\!\mu^2}}{\mu} |\re \lambda | \Bigr\} \scup \Bigl[-\dfrac{\beta}{2},0 \Bigr)
    & \mbox{if }\gamma\!<\!\infty,\, \mu^2\!\ge\!\gamma\delta.
  \end{cases}
\]
\end{enumerate}
\end{prop}

\begin{proof}
(a) The first claim for $\mu \ge 2$ is immediate from Theorem~\ref{nett} (ii) and \eqref{Paderiv};
the second claim follows if we additionally use Theorem \ref{hendrik} and observe that $\gamma \ge \beta$ by \eqref{ineq} and hence
$\gamma > \frac \beta 2 \Bigl( 1 + \sqrt{1\!-\!\frac 4{\beta\delta}} \,\Bigr)$.

(b)
The first claim for $0<\mu<2$ follows from Theorem~\ref{nett} (i), (ii) and \eqref{Paderiv}.
%
It remains to consider the case $0<\mu<2$ and $\delta>0$.
First we determine if, for $\re\lambda \in \bigl( - \frac 2\delta, - \frac \beta 2 \bigr]$,
the boundary of the sector in \eqref{Paderi} intersects the circle $(\re \lambda + \frac 1\delta)^2 + (\im \lambda)^2 = \frac 1{\delta^2}$ in \eqref{Paderiii}.
The imaginary part of boundary points of the sector equals $\pm \frac{\sqrt{4-\mu^2}}\mu |\re \lambda|$
and so points $\lambda$ of the intersection satisfy
\[
 \Bigl(\re \lambda + \frac{1}{\delta}\Bigr)^{\!2}  + \frac{4-\mu^2}{\mu^2} (\re \lambda)^2 = \frac{1}{\delta^2}\,.
\]
A simple calculation yields $\re \lambda = \smash[t]{-\frac{\mu^2}{2\delta}}$.
Observe that $- \frac 2 \delta <  -\frac{\mu^2}{2\delta} \le -\frac\beta2$ since $\mu<2$ and $\mu^2\ge\beta\delta$, see \eqref{ineq}.
Now Theorem~\ref{nett} (i), (ii), and (iii) implies all the claims for $\gamma=\infty$.
For $\gamma<\infty$, we additionally use \eqref{Paderiv} and recall that $\gamma \ge \beta$
by \eqref{bgdm}; then $-\frac\gamma2 \le -\frac\beta2$, and it remains to note that
$-\!\min\bigl\{\frac 2\delta,\frac \gamma 2\bigr\} < -\frac{\mu^2}{2\delta}$ if and only if $\mu^2<\gamma\delta$.
\end{proof}

\section{Application: Small transverse oscillations \\of an ideal incompressible fluid in a pipe}\label{ex}   

\renewcommand{\d}{{\rm d}}

The small transverse oscillations of a horizontal pipe of length normalized to~1 carrying a steady-state flow of an ideal incompressible fluid
are described~by
\begin{equation}\label{beam}
   \frac{\partial ^2 u } {\partial t^2 } + \frac{\partial
   ^2}{\partial r^2 } \left [ E  \frac{\partial ^2 u } {\partial r^2
   } + {C} \frac{\partial ^3 u }{\partial r^2 \partial t } \right] +
K \frac{\partial^2 u}{\partial t\partial r} =0,
   \hspace{2em} r \in (0,1), \ t > 0,
\end{equation}
see e.g.\ \cite{shk94}.
Here $u(r,t) $ denotes the transverse displacement at time $t$ and
position $r$, and $E$, $C$, $K$ are positive physical constants.
The last term on the left hand side of (\ref{beam})
is called the gyroscopic term.
If the pipe is pinned at both endpoints, the boundary conditions
\begin{equation}
u\big|_{r=0} =\   0,  \quad
 \frac{\partial^2 u }{\partial r^2}\bigg|_{r=0}  =\ 0,  \quad
  u\big|_{r=1} =\   0,  \quad
 \frac{\partial^2 u }{\partial r^2}\bigg|_{r=1}  =\ 0
    \label{bcs}
\end{equation}
have to be imposed at any time $t>0$.

The partial differential equation \eqref{beam} with boundary conditions \eqref{bcs} is a se\-cond order problem \eqref{sys}
in the Hilbert space $H \!=\! L^2 (0,1)$. Here the operator $A_0$ in $H$ is given~by
\begin{align*}
   A_0 = E \frac{\d^4}{\d r^4},
   \quad
   \dom(A_0) = \left\{ z \in H^4(0,1) \mid z(0)=
   z(1)=  z'' (0)= z'' (1) =0 \right\},
\end{align*}
where $H^4(0,1)$ is the fourth order Sobolev space associated with $L^2(0,1)$.
Clearly, $A_0$ satisfies assumption (A1), $A_0^{-1}$ is a compact operator, and
\[
 A_0^{\frac 12} = -\sqrt{E}\frac{\d^2}{\d r^2}, \quad  H_{\half} = \dom (A_0^{\frac 12}) = \left\{ z \in H^2(0,1)\mid z(0)=z(1) =0 \right\},
\]
with inner product and norm on $H_{\half}$ given by
\begin{equation}
\label{Delft}
\la z, v \ra_{\hhalf} = E \la z'', v'' \ra, \quad \|z\|_{\hhalf}^2 \geq E \pi^4 \|z\|^2,
\quad z, v \in  H_{\frac{1}{2}},
\end{equation}
i.e.\  $a_0=\sqrt{E}\pi^2$. The damping operator $D$ defined as
\begin{align*}
   D= C \frac{\d^4}{\d r^4} + K \frac{\d}{\d r} =\frac CE  A_0 + K \frac{\d}{\d r}: H_{\half} \to H_{-\half}
\end{align*}
is bounded and maps $\dom(A_0)$ into $H$. Moreover,
for $z\in H_\half$,
\begin{equation}
\label{eqn1}
  \re \langle Dz,z\rangle_{H_{-\half}\times H_{\half}}
   = C \langle z'',z''\rangle
   = \frac CE \|z\|^2_{\hhalf}
   \ge  \frac C{\sqrt{E}} \pi^2 \|z\|_{\hhalf} \|z\|
    \ge C \pi^4 \|z\|^2.
\end{equation}
Thus assumptions (A2) and (A3) hold as well. However, $D$ is not self-adjoint due to the first order derivative coming from the gyroscopic term in~\eqref{beam}.

From \eqref{eqn1} we obtain the following information on the constants in the spectral enclosures in Theorem~\ref{supernett} which were defined at the beginning of Section~\ref{section3}.

\begin{prop} \label{NEW22}
For the operator $D$, we have \vsqueeze{1mm}
\begin{equation*}\label{Numbers}
\beta = C \pi^4, \quad \gamma = \infty, \quad \delta =\frac CE\,, \quad
\mu =  \frac C{\sqrt{E}}  \pi^2,
\vspace{-3mm}
\end{equation*}
and one can choose 
\[
  k = \frac{K}{C\pi^3}\,.
\]
\end{prop}

\begin{proof}
\begingroup
\thickmuskip=5mu plus 4mu minus 1mu
From \eqref{eqn1} we obtain $\beta \ge C \pi^4$, $\gamma = \infty$, $\delta = \frac CE$ and $\mu \ge \smash{\frac{C}{\sqrt{E}}} \pi^2$.
Since in~\eqref{eqn1} equality holds everywhere if we choose $z=z_0$ where $z_0(t)=\sin (\pi t)$, $t\in[0,1]$, is
the eigenfunction of $A_0^{\frac 12}$ corresponding to its smallest~eigen\-value $\pi^2\sqrt{E}$, the equalities $\beta = C \pi^4$, $\mu = \frac C{\sqrt{E}} \pi^2$~follow.
\endgroup

To prove the last claim, we let $z \in H_{\half}$ and estimate
\begin{equation*} \label{Sobolev}
\|z'\|^2 = \la z',z' \ra =  -\la z'',z \ra \le \|z''\| \|z\|.
\end{equation*}
Using this estimate, $\|z\| \le \frac{1}{\pi^2} \|z''\|$ and \eqref{eqn1}, we conclude that
\[
\bigl|\im\la Dz,z\ra_{H_{-\half}\times H_{\half}}\bigr|
  =  K \bigl| \langle z',z\rangle \bigr|
 \leq K \|z''\|^{1/2} \|z\|^{3/2} \le \frac{K}{\pi^3} \|z''\|^2
  = \frac{K}{C\pi^3} \re \la Dz,z \ra_{H_{-\half}\times H_{\half}}.
  \qedhere
\]
\end{proof}
\vspace{0pt}

\begin{thm} \label{NEW222}
The spectrum of the operator $\cA$ given by \eqref{AA1} associated with the boundary value problem \eqref{beam}, \eqref{bcs}
satisfies the inclusion
\[
  \sigma(\cA)
  \subset
  \left\{ \lambda \in \C  \mid \re \lambda < 0, \ \Bigl| \re \lambda - \frac \beta 2 \Bigr|^2 \mkern-1.5mu \ge
	\Bigl( \frac \beta 2 \Bigr)^{\!2} \Bigl( 1 - \frac 4 {\beta \delta} \Bigr), \
 |\im \lambda| \le h(|\re\lambda|)
\right\},
\vspace{-2mm}
\]
where
\[
  h(t) =
\begin{cases}
\hspace{4mm} \dfrac {kt}{ 1- \frac 2\beta t}\,, & \ 0 \le t < \lambda_{\rm i,ii}, \\[2.5ex]
\displaystyle \sqrt{\frac \beta\delta \frac{t}{\beta-t}-t^2}\,, &
\ \lambda_{\rm i,ii} \le t \le \lambda_{\rm ii,iii}, \\
\hspace{8mm} h_{\rm iii}(t),
& \ \lambda_{\rm ii,iii} < t <\infty,
\end{cases}
\]
with $k_\mu$, $h_{\rm iii}$ as defined in Theorem~{\rm \ref{supernett}},
$\lambda_{\rm i,ii} = \frac \beta 2 \bigl( 1 - \frac k{k_\mu} \bigr)$,
$\lambda_{\rm ii,iii} = \frac \beta 2 \bigl( 1 + \frac k{k_\mu}\rule[-1.4ex]{0pt}{0pt} \bigr)$,
and the constants $\beta$, $\delta$, $\mu$, $k$ as defined in Proposition {\rm \ref{NEW22}};
in particular, there is a spectral free strip if $\beta\delta>4$,
\[
  \Re \sigma(\cA) \cap \biggl( \!- \frac{C\pi^4}2 \biggl( 1+\sqrt{1\!-\!\frac{4E}{C^2\pi^4}}\,\biggr),
	- \frac{C\pi^4}2 \biggl( 1-\sqrt{1\!-\!\frac{4E}{C^2\pi^4}}\,\biggr)\biggr) = \emptyset \quad \mbox{ if } \  \ C > \frac{2\sqrt{E}}{\pi^2}\,.
\]
\end{thm}

\begin{proof}
All claims follow from Proposition~\ref{compare}
and Remark \ref{newrem} if we note that here $\beta \delta = \mu^2$ whence
$\lambda_{\rm ii,iii}$ has the claimed form and $I_{0,\mu} = \bigl(\lambda_{\rm i,ii},\lambda_{\rm ii,iii}\bigr)$.
The form of the spectral free strip $|\re\lambda| \notin I_0$ is obtained by inserting the constants from
Proposition~\ref{NEW22} into \eqref{I0}.
\end{proof}

\begin{example}\label{example}
For the physical constants
\[
  E=25, \quad C=1, \quad K=14
\vspace{-3mm}
\]
one can compute that
\[
    \lambda_{\rm i,ii} \approx 19.859, \qquad
    \lambda_{\rm ii,iii} \approx 77.550;
\]
the corresponding spectral inclusion in Theorem \ref{NEW222} is displayed in Figure~\rmref{fig8}.
Note that here Theorem~\rmref{NEW222} does not yield a spectral gap since $C = 1 < 10/\pi^2 = 2\sqrt E/\pi^2$.
If we increase $C$ to the critical value $2\sqrt{E}/\pi^2$, i.e.\ if we choose
\[
  E=25, \quad \smash{C=\frac{10}{\pi^2}}, \quad K=14,
\vspace{-3mm}
\]
then
\[
    \lambda_{\rm i,ii} \approx 19.852, \qquad
    \lambda_{\rm ii,iii} \approx 78.844;
\]
Figure~\rmref{fig9} shows the corresponding spectral inclusion in Theorem \ref{NEW222} right before the opening of the spectral free strip.
\begin{figure}%
\newcommand\numrange[2]{%
\def\numr{-#1*x+2*sqrt(-x/#2)}
\begin{scope}[very thin,color=lightgray,fill=lightgray]
  \addplot[domain=-110:-0.01,samples=100] {-(\numr)} \closedcycle;
  \addplot[domain=-110:-0.01,samples=100] {\numr} \closedcycle;
\end{scope}
  \addplot[domain=-110:-0.01,samples=100,color=gray] {-(\numr)};
  \addplot[domain=-110:-0.01,samples=100,color=gray] {\numr};
  \draw[-] (axis cs:-110,-80) -- (axis cs:0,-80);
  \draw[-] (axis cs:-110,80) -- (axis cs:0,80);
  \draw[-] (axis cs:-110,-80) -- (axis cs:-110,80);
}
\thickmuskip=4mu plus 3mu minus 2mu
\setlength{\captionmargin}{3mm}
\begin{floatrow}
\ffigbox[\FBwidth]{\caption{\newline
Example~\ref{example} for physical parameters $E=25$, $C=1$, $K=14$.}
\label{fig8}}
{
\begin{tikzpicture}
\def\onetwo{-19.8593}
\def\twothree{-77.5498}
\def\hnull{sqrt(-2435.227*x/(97.4+x)-x^2)}
\def\hi{0.4515215*x/(1+0.020532*x)}
\def\hiii{0.4175*x-26.75}
\begin{axis}[xmin=-110,xmax=15,ymin=-80,ymax=80]
\pgfplotsset{ticks=none}
\numrange{0.45152}{0.04}
\begin{scope}[very thin,color=red,fill=red]
  \addplot[domain=\onetwo:-0.01,samples=30] {-\hi} \closedcycle;
  \addplot[domain=\onetwo:-0.01,samples=30] {\hi} \closedcycle;
  \addplot[domain=\twothree:\onetwo,samples=30] {-\hnull} \closedcycle;
  \addplot[domain=\twothree:\onetwo,samples=30] {\hnull} \closedcycle;
  \addplot[domain=-110:\twothree,samples=10] {\hiii} \closedcycle;
  \addplot[domain=-110:\twothree,samples=10] {-(\hiii)} \closedcycle;
\end{scope}
  \addplot[domain=\onetwo:-0.01,samples=30] {-\hi};
  \addplot[domain=\twothree:\onetwo,samples=30] {-\hnull};
  \addplot[domain=-110:\twothree,samples=10] {\hiii};
  \addplot[domain=-110:\twothree,samples=10] {-(\hiii)};
  \addplot[domain=\twothree:\onetwo,samples=30] {\hnull};
  \addplot[domain=\onetwo:-0.01,samples=30] {\hi};
  \addplot[domain=-110:15,samples=2] {0};
\draw[thick, -] (axis cs:-1,50) -- (axis cs:1,50);
\draw[-] (axis cs:0,-80) -- (axis cs:0,80);
\draw[thick, -] (axis cs:\onetwo,-1.5) -- (axis cs:\onetwo,1.5);
\draw[thick, -] (axis cs:\twothree,-1.5) -- (axis cs:\twothree,1.5);
\pgfplotsset{
   after end axis/.code={
        \node[right] at (axis cs:0,50){ $50$ };
        \node[below] at (axis cs:\onetwo,-1){ \llap{$-$}$19.9$ };
        \node[below] at (axis cs:\twothree,-1){ $-77.5\;$ };
				}}
    \end{axis}
\end{tikzpicture}
}
\hspace{3mm}
\ffigbox[\FBwidth]{\caption{\newline
Example~\ref{example} for physical parameters $E=25$, $C=\frac{10}{\pi^2}$, $K=14$.}
\label{fig9}}
{
\begin{tikzpicture}
\def\onetwo{-19.8518}
\def\twothree{-78.8442}
\def\hnull{sqrt(-2435.227*x/(98.696+x)-x^2)}
\def\hi{0.445634*x/(1+0.020264*x)}
\def\hiii{0.4108*x-26.4}
\def\samples{89}
\begin{axis}[xmin=-110,xmax=15,ymin=-80,ymax=80]
\pgfplotsset{ticks=none}
\numrange{0.44563}{0.04053}
\begin{scope}[very thin,color=red,fill=red]
  \addplot[domain=\onetwo:-0.01,samples=30] {-\hi} \closedcycle;
  \addplot[domain=\onetwo:-0.01,samples=30] {\hi} \closedcycle;
  \addplot[domain=\twothree:\onetwo,samples=\samples] {-\hnull} \closedcycle;
  \addplot[domain=\twothree:\onetwo,samples=\samples] {\hnull} \closedcycle;
  \addplot[domain=-110:\twothree,samples=10] {\hiii} \closedcycle;
  \addplot[domain=-110:\twothree,samples=10] {-(\hiii)} \closedcycle;
\end{scope}
  \addplot[domain=\onetwo:-0.01,samples=30] {-\hi};
  \addplot[domain=\twothree:\onetwo,samples=\samples] {-\hnull};
  \addplot[domain=-110:\twothree,samples=10] {\hiii};
  \addplot[domain=-110:\twothree,samples=10] {-(\hiii)};
  \addplot[domain=\twothree:\onetwo,samples=\samples] {\hnull};
  \addplot[domain=\onetwo:-0.01,samples=30] {\hi};
  \addplot[domain=-110:15,samples=2] {0};
\draw[thick, -] (axis cs:-1,50) -- (axis cs:1,50);
\draw[-] (axis cs:0,-80) -- (axis cs:0,80);
\draw[thick, -] (axis cs:\onetwo,-1.5) -- (axis cs:\onetwo,1.5);
\draw[thick, -] (axis cs:\twothree,-1.5) -- (axis cs:\twothree,1.5);
\pgfplotsset{
   after end axis/.code={
        \node[right] at (axis cs:0,50){ $50$ };
        \node[below] at (axis cs:\onetwo,-1){ \llap{$-$}$19.9$ };
        \node[below] at (axis cs:\twothree,-1){ $-78.8\;$ };
				}}
    \end{axis}
\end{tikzpicture}
}
\end{floatrow}
\vspace{6mm}
\begin{center}
Figures \ref{fig8}, \ref{fig9}: Spectral enclosures obtained from 
$W(\cA)$ (light grey) and \\ from $W^2(\cA)$ (\redgrey).
\end{center}
\end{figure}
\end{example}

\bigskip

\noindent
{\bf Acknowledgements.}
The second author gratefully acknowledges the support of the \emph{Swiss National Science Foundation}, SNF, grants no.\ $200020\_146477$ and $169104$.

{\small

\bigskip

\noindent
\parbox[t]{9cm}{\small
Birgit Jacob \\
Fakult\"at f\"ur Mathematik und Naturwissenschaften\\
Fachgruppe Mathematik und Informatik\\
Bergische Universit\"at  Wuppertal\\
Gau\ss stra\ss e 20\\
D-42119 Wuppertal\\
Germany\\
Tel.: +49-202-439-2527\\
Email: jacob@math.uni-wuppertal.de\\
}
\parbox[t]{6.7cm}{\small
Christiane Tretter \\
Mathematisches Institut \\
Universit\"{a}t Bern\\
Sidlerstra\ss e 5 \\
CH-3012 Bern\\
 Switzerland\\
Tel.: +41-31-631-8820\\
Email: tretter@math.unibe.ch
}

\ \\
\noindent
\parbox{9cm}{\small
Carsten Trunk \\
Institut f\"ur Mathematik \\
Technische Universit\"at Ilmenau \\
Postfach 100565 \\
D-98684 Ilmenau \\ Germany \\ Tel.: +49-3677-69-3253 \\
Fax: +49-3677-69-3270  \\Email: carsten.trunk@tu-ilmenau.de
}
\parbox{6.7cm}{\small
Hendrik Vogt\\
Fachbereich 3 - Mathematik \\
Universit\"at Bremen\\
Postfach 330 440\\
Bibliothekstra\ss e 1\\
D-28359 Bremen \\
Germany\\
Tel.: +49-421-218-63702\\
Email:
hendrik.vo\rlap{\textcolor{white}{hugo@egon}}gt@uni-\rlap{\textcolor{white}{%
hannover}}bremen.de}

}  
\bigskip


\end{document}